\documentclass[11pt,a4paper]{article}

\usepackage[a4paper,top=3.8cm,bottom=3.8cm,left=3cm,right=3cm]{geometry}%
\usepackage{hyperref}

\usepackage[T1]{fontenc}
\usepackage{amsmath,amssymb,mathtools,amsthm,amsopn}

\usepackage{tocloft}
\usepackage{fancyhdr}
\usepackage{fourier-orns}

\usepackage{cite}

\usepackage{mathrsfs}
\renewcommand{\phi}{\varphi}
\newcommand{\perm}{\mathfrak{S}}

\newcommand{\I}{\mathcal{I}}
\newcommand{\step}[1]{\par\medskip\noindent\it#1\rm}

\newcommand{\dcc}{d_{\textup{cc}}}
\DeclareMathOperator{\espo}{e}
\newcommand{\eap}{\espo_{\textup{ap}}}
\newcommand{\expap}{\exp_{\textup{ap}}}

\newcommand{\abs}[1]{\lvert#1\rvert}

\newcommand{\g}{\gamma}

\renewcommand{\H}{\mathcal{H}}
\newcommand{\e}{\varepsilon}
\renewcommand{\r}{\varrho}
\newcommand{\B}{\mathcal{B}}

\newcommand{\s}{\sigma}

\renewcommand{\P}{\mathcal{P}}

\newcommand{\la}{\lambda}

\renewcommand{\cal}[1]{\mathcal{#1}}

\newcommand{\wt}{\widetilde}

\newcommand{\A}{\mathcal{A}}
\renewcommand{\d}{\delta}
\newcommand{\Eucl}{\textup{Euc}}
 
\newcommand{\p}{\partial}

\DeclareMathOperator{\diag}{diag}

\newtheoremstyle{pluto}  {}{}
{\slshape}  {}{\bfseries}  {.} {1ex}    {}

\newtheorem{theorem}{Theorem}[section]
\newtheorem{proposition}[theorem]{Proposition}

\newtheorem{lemma}[theorem]{Lemma}

\theoremstyle{pluto}
\newtheorem{definition}[theorem]{Definition}
\newtheorem{remark}[theorem]{Remark}
\newtheorem{example}[theorem]{Example}

\usepackage{enumitem}
\newenvironment{enumerate*}{\begin{enumerate}[noitemsep] }{\end{enumerate}}
\newenvironment{itemize*}{\begin{itemize}[noitemsep] }{\end{itemize}}
\newenvironment{description*}{\begin{description}[noitemsep] }{\end{description}}

\newcommand{\xls}{x_{\textup{LS}}}
\newcommand{\R}{\mathbb{R}}
\newcommand{\N}{\mathbb{N}}
\renewcommand{\o}{\omega}
\renewcommand{\d}{\delta}
\renewcommand{\t}{\tau}

\newcommand{\D}{\Delta}

\renewcommand{\a}{\alpha}
\renewcommand{\b}{\beta}
\newcommand{\loc}{\textup{loc}}
\DeclareMathOperator{\Lip}{Lip}
\DeclareMathOperator{\Span}{span}
\DeclareMathOperator{\ad}{ad}
\newcommand{\W}{\mathcal W}

\newcommand{\norm}[1]{\left\Vert#1\right\Vert}
\numberwithin{equation}{section}

\allowdisplaybreaks[3]
\let\oldbibliography\thebibliography
\renewcommand{\thebibliography}[1]{%
  \oldbibliography{#1}%
  \setlength{\itemsep}{0pt}%
}

\usepackage{titlesec}
\titleformat{\section}{%
\normalfont\large\bfseries}{\thesection.}{1em}{}
\titleformat{\subsection}{%
\normalfont\normalsize\bfseries}{\thesubsection.}{1em}{}

\begin{document}

\title{Generalized Jacobi identities and ball-box theorem
\\ for horizontally regular vector fields}
\author{Annamaria Montanari \and Daniele Morbidelli}

\date{}

\maketitle

\begin{abstract}
Consider a family $\H:= \{X_j =: f_j\cdot\nabla: j=1,\dots, m\}$ of $C^1$ vector 
fields in
$\R^n$ and let $s\in\N$. We assume that for all $p\in\{1.\dots, s\}$ and $j_1, \dots, j_p\in \{1, \dots, m\}$
the \emph{horizontal derivatives} $X_{j_1}X_{j_2}\cdots X_{j_{p-1}}f_{j_p}$
exist and are Lipschitz continuous with respect to the control distance defined by~$\H$.
Then we show that different notions of commutator agree.
This involves an accurate analysis of some algebraic identities involving nested commutators which seem to have an independent interest. 

Our principal applications are  a ball-box theorem, the doubling property and the Poincar\'e inequality for H\"ormander vector fields under an intrinsic ``horizontal regularity'' assumption on their coefficients.
\end{abstract}

\tableofcontents

\section{Introduction and main results}
In this paper we study the notion of higher order commutator for a  given 
family $\H= \{X_1, \dots, X_m\}$ 
 of vector fields in $\R^n$.
Our main issue   is to discuss to what extent the notion of higher
order
commutator can be
extended to
vector fields $X_j\in C^1_\Eucl$   whose higher order
derivatives are assumed to have regularity only
along the ``horizontal directions'' provided by the family $\H$.
The main application of such study consists of a 
discussion of a class of    \emph{almost exponential maps}  
 under very low,
 intrinsic regularity assumptions which is carried out in\cite{MontanariMorbidelli11d}. 
This enables   us to 
prove a boll-box theorem,  the doubling property and the Poincar\'e 
inequality for vector fields satisfying the H\"ormander's bracket-generating  condition of step $s\ge 1$ under very low regularity requirements.

To understand the problem, which starts to appear for
commutators of length three,  assume that $X_1,
\dots, X_m$ are vector fields of class $C^1_\Eucl$, i.e.~of class $C^1$ in the Euclidean sense.
Write   $X_i= f_i\cdot\nabla$ for $i=1,\dots,m$.
The definition of commutarors of length two is clear, namely we set
\begin{equation*}
X_{jk}  := (X_j^\sharp f_k - X_k^\sharp f_j)\cdot\nabla  = :
f_{jk}\cdot\nabla
\quad\text{for all $j,k\in\{1,\dots,m\}$}
\end{equation*}
and $X_{jk}\psi: = f_{jk}\cdot\nabla\psi$ for $\psi\in C^1_\Eucl(\R^n)$.
Here we 
denote by  
\(X^\sharp f(x):=\lim_{t\to 0}\frac 1t(f(e^{tX}x)-
f(x))\)
the Lie derivative along a   vector field $X\in C^1_\Eucl$ of a scalar
funciton $f$ (such unusual
 notation will be convenient for our purposes).

Passing to length three,
 we have two alternatives.  For each  $i,j,k$, we can define either
\begin{align}\label{zerosei}
  X_{ijk}:= [X_i, [X_{j}, X_k]]  &:= (X_i^\sharp X_j^\sharp  f_k - X_i^\sharp
X_k^\sharp f_j - X_j^\sharp  X_k^\sharp
f_i   + X_k^\sharp  X_j^\sharp  f_i)\cdot\nabla,
\shortintertext{or} \label{zerotto}
 \ad_{X_i} X_{jk} &:= (X_i^\sharp f_{jk} - X_{jk} f_i)\cdot\nabla.
\end{align}
Both operators act on $C^1_\Eucl$ functions.
The first one is the most natural
and symmetric (for instance one gets the Jacobi identity for free). The second one appears in some
useful non commutative calculus formulae which play a key role in our work,
see Theorem \ref{seminuovo}. It is
rather easy to see that $[X_i, [X_{j}, X_k]]  = \ad_{X_i}X_{jk}$, if the
involved vector
fields are $C^2_\Eucl$,  so
that Euclidean second order derivatives commute (here end hereafter by $C^k_\Eucl$ we
denote Euclidean $C^k$ regularity). In this paper we are able to
show
that
this regularity is not necessary. Indeed, if we denote by $C^{2,1}_{\H,\loc}$
all
functions $f$  which have two horizontal derivatives and such that for
all $i,j\in\{1,\dots,m\}$ the function  $X_i^\sharp
X_j^\sharp f$ is locally Lipschitz with respect to the distance associated to
the vector fields in $\H$, then we have
\begin{theorem}[see Theorem \ref{dass} for a higher order statement]
\label{tregi}  Let $\H = \{X_1, \dots, X_m\}$ be a family of $C^1_\Eucl$ vector
fields. Write $X_j = f_j\cdot\nabla$ and
assume that $f_j\in C^{2,1}_{\H, \loc}$ for all $j\in\{1,\dots, m\}$.
Then we
have
\begin{equation*}
[X_i, [X_{j}, X_k]]  = \ad_{X_i} X_{jk}\quad \text{for all $i,j,k\in \{1,
\dots,
m\}$. }\end{equation*}
\end{theorem}
In Theorem \ref{dass} we  give the general version of this statement which involves commutators of 
arbitrary length $s\ge 1$, where the vector fields  $f_j$ belong to the class $ C^1_\Eucl\cap C^{s-1,1}_{\H,\loc}$ introduced in Definition \ref{montgomery}.

Although the analysis of statements like Theorem \ref{tregi}---together with the techniques we develop in the proof---may have some independent interest, the strong motivation why we need to analyze operators like 
$\ad_{X_i}X_{jk}$ comes from their natural appearance in the noncommutative formulas of Theorem 
\ref{seminuovo} and ultimately in the theory of differentiation of the almost exponential maps
$E$ introduced below. 
Let us mention that in \cite{MontanariMorbidelli11d} we prove a higher order   orbit theorem for    families $\H$ in such 
class; \cite[Example~3.14]{MontanariMorbidelli11d} shows that our regularity classes capture examples which do not fall in the classical framework.

In order to show that  \eqref{zerosei} and
\eqref{zerotto} agree,
we need to analyze carefully the algebraic properties of the coefficients
appearing in the expansion of nested commutators as sums of higher order
derivatives.
In particular, we exploit some higher order algebraic identities which we denote
as
 ``generalized
Jacobi identities''; see  Proposition \ref{genero} and see also 
Proposition \ref{postace}.
It is  
interesting to observe that some of those  identities, specialized to particular
situations, give the proof of some old nested commutators identities going back
to Baker and discussed   in    \cite{Oteo}.
This is
discussed in
Subsection \ref{chetbaker}. We believe that these algebraic features
may have some independent interest.

From an historical point of view, let us mention that for commutators of
length two,  notions of nonsmooth Lie brackets have been studied deeply by
Rampazzo and Sussmann \cite{RamSus}. The notion of set-valued
commutator studied in
\cite{RamSus} concerns vector fields which are quite less regular than ours,
actually  Lipschitz continuous only, but this approach does not provide
a quantitative knowledge of   control balls or Poincar\'e inequalities. Moreover,
the notion of set-valued commutator is not clear,  if  the length exceeds two; see  
the counterexample in \cite[Section 7.2]{RamSus}. We work here at a
slightly more comfortable level of
regularity, which ensures that commutators are pointwise defined and
``horizontally'' Lipschitz continuous, which will be sufficient to obtain some
good information on control balls.

Next we discuss our applications to sub-Riemannian geometry.
 Let  $\H=\{X_1, \dots, X_m\}$ 
 be a family of   vector fields and assume that $X_j\in C^1_\Eucl\cap C^{s-1, 1}_{\H, \loc}$
for some $s\in\N$. Denote by $B_\textup{cc}(x,r)$ the Carnot--Carath\'eodory ball with center at $x$ and radius $r$.  Let  $\P:= \{Y_1, \dots, Y_q\}$ be the family of all nested commutators of length at most $s$.
Let $\text{length}(Y_j)=:\ell_j
\le  s$.
 Assume that 
$\H$ satisfies the H\"ormander  condition of step $s$, i.e.~$\dim\Span\{Y_j(x)\}=n$, for all $x\in\R^n$.
To identify a family of $ n$ commutators, let us choose  a multiindex $I=
(i_1, \dots, i_n)\in \{1, \dots, q\}^n$.
Given a radius  $r>0$, define  the scaled    commutators $\wt Y_{i_k}:=
r^{\ell_{i_k}}Y_{i_k}$ and the
\emph{almost  exponential map}
\begin{equation}\label{expoenne}
 E_{I,x,r}(h):= \expap(h_1  \wt Y_{i_1})\cdots  \expap(h_p
 \wt Y_{i_n})x
\end{equation}
for each $h$ close to $0\in \R^n$ (after passing to $\wt Y_{i_j}$, the variable
$h$ lives at a unit scale). See~\eqref{hhh} for the definition of    the \emph{approximate exponential} $\expap$.  
Below,   $B_\r$ denotes the control
ball defined by
\emph{all} commutators (with their degrees, see \eqref{coscos}), which trivially
contains the Carnot--Carath\'eodory  ball $B_{\textup{cc}} $ with same center
and radius 
defined  in \eqref{dicacco}.
Then we have the following  ball-box theorem and Poincar\'e inequality.
A more detailed statement is contained in Section~\ref{cinquemila}.

\begin{theorem}\label{pocaro}
Let $\H$ be a family of vector fields 
 in the class $C^1_\Eucl\cap C^{s-1, 1}_{\H, \loc}$ for some~$s$. Assume the H\"ormander 
condition of step $s$ and assume that $Y_j\in C^0_\Eucl$
for all $Y_j\in\cal{P}$. Let $\Omega\subset\R^n$ be a
bounded set. Then there is $C>1$ such that the following holds.  Let $x\in
\Omega$ and take a  positive radius $r<C^{-1}$.
Then there is   a subfamily 
$\{Y_{i_1}, \dots, Y_{i_n}\}\subset\P$ such that the map $E:= E_{I,x,r}$ in \eqref{expoenne}
is $C^1$ in the Euclidean sense  on the  unit ball
$B_\Eucl(1)\subset\R^n$. Its Jacobian satisfies the estimate $C^{-1}\abs{\det dE(0)}\le
\abs{\det dE(h)}\le C\abs{\det dE(0)}$ and we have the ball-box inclusion
\begin{align}
\label{quickbuild}
 E(B_\Eucl(1))& \supseteq B_{\r}(x, C^{-1}r).
\end{align}
 The map $E$ is one-to-one
 on $B_\Eucl( 1)$ and
we have
\begin{equation}\label{cicanto} 
      \abs{B_\textup{cc}(x, 2r)}\le C\abs{B_\textup{cc}(x,r)} \quad\text{for all $x\in\Omega$ 
$r<C^{-1}$.}
\end{equation}
Moreover,   for any $C^1$ function $f$ we have the Poincar\'e inequality
\begin{equation}\label{1122} 
 \int_{B_{\textup{cc}}(x,r)}\abs{f(y) - f_{B_{\textup{cc}}(x,r)}} d y\le
C\sum_{j=1}^m\int_{B_{\textup{cc}}(x, Cr)} \abs{ rX_jf(y)} d y.
\end{equation}
\end{theorem}

It is well known that the doubling estimate and the  Poincar\'e inequality 
are important tools in subelliptic PDEs,  sub-Riemanninan geometry and analysis in metric spaces; see 
\cite{FranchiLanconelli83,NagelSteinWainger,Jerison,SaloffCoste92,GarofaloNhieu96,Cheeger
,HajlaszKoskela}. Note that inequality \eqref{1122} improves all previous  versions of the Poincar\'e inequality
from a regularity standpoint; compare
 \cite{Jerison,LM,BramantiBrandoliniPedroni08,MM,Manfredini}.  
Indeed, in such papers some higher
 order Euclidean regularity were assumed, whereas our higher order derivatives
$X_{j_1}^\sharp\cdots X_{j_{p-1}}^\sharp f_{j_p}$ with $2\le p\le s$ are assumed to be horizontally Lipschitz 
continuous only. 
  \footnote{Technically speaking, 
 both   the approaches adopted in  in  \cite{BramantiBrandoliniPedroni08} 
and \cite{MM}---via Euclidean Taylor approximation or   Euclidean regularization---do not work in our situation, because one cannot prove that higher order commutators of mollified vector fields converge to mollified of the corresponding   commutators.}

Let us mention that in  \cite{MontanariMorbidelli11d} and \cite{MontanariMorbidelli11c}, relying on the results obtained here, 
we also  prove an integrability result for orbits, a ball-box theorem and the Poincar\'e inequality  in
a setting where the H\"ormander's condition  is removed.

Our work in low regularity is also motivated by the appearance, in several recent papers,  of 
subelliptic PDEs involving nonlinear first order operators. 
This happens for instance in several complex variables, while studying graphs with prescribed Levi curvature in $\mathbb{C}^n$, 
see \cite{CittiLanconelliMontanari02},  or in the study of intrinsic regular hypersurfaces  in Carnot groups, see \cite{AmbrosioSerraCassanoVittone06}, where vector fields with non Euclidean regularity naturally appear.
 These papers suggest that it
 would be
  desirable to   remove even our assumption $X_j\in C^1_\Eucl$ for the vector fields of the horizontal 
family $\H$.
However,  note that removing such assumption   would destroy uniqueness of integral curves and   dealing efficiently with our almost exponential maps would require nontrivial new ideas.

Before closing this introduction, we mention some recent papers where
nonsmooth vector fields are discussed. In \cite{SawyerWheeden}, the case
of diagonal vector fields is discussed deeply.
In the  H\"ormander case, in the model situation
of equiregular
 families of vector fields,
nonsmooth ball-box theorems have been studied by 
see \cite{KarmanovaVodopyanov,Greshnov}. Finally, \cite{BramantiBrandoliniPedroni2} contains a
nonsmooth lifting theorem.

\paragraph{Acknowledgements.}
We thank Francesco Regonati, who helped us to formalize some of the questions
we encountered in Section \ref{identitone} in the language of polynomial
identities.

\section{Preliminary facts on horizontal regularity}
\subsection{Horizontal regularity classes}
\label{preliminarmente}

\paragraph{Vector fields and
the control distance.}
Consider a
family  $\H=\{ X_1, \dots, X_m\}$  of vector fields
 and assume that   $X_j\in C^1_\Eucl(\R^n)$ for all $j$. Here and later
$C^1_\Eucl$ means $C^1$ in the Euclidean sense.
Write $X_j = :f_j\cdot\nabla$, where $f_j\colon\R^n\to
\R^n$.
The vector field $X_j$, evaluated at a point $x\in \R^n$, will be denoted
by  $X_{j,x}$ or $ X_j(x)$.
All the vector fields in this paper are always defined on  the whole
space $\R^n$.

Define the Franchi--Lanconelli distance \cite{FL}
\begin{equation}\label{dicocco}
\begin{aligned}
 d(x,y)&:=
\inf \Big\{r>0: \text{$y=e^{t_1 Z_1}\cdots e^{t_\mu Z_\mu} x$ for some $\mu\in
\N$}
 \\& \qquad\qquad \qquad
\text{where $\sum \abs{t_j}\le 1$
 with $Z_j\in  r\cal{H} $} \Big\}.
\end{aligned}
\end{equation}
Here  and hereafter we let $r\H:=  \{r X_1,\dots, r X_m\}$ and
$\pm r\H:=  \{\pm r X_1,\dots, \pm r X_m\}$.

Let also $\dcc$ be the Fefferman--Phong  and Nagel--Stein--Wainger distance \cite{FeffermanPhong81,NagelSteinWainger}
\begin{equation}\label{dicacco}
\begin{aligned}
\dcc & (x,y) := \inf \big\{r>0: \text{ there is $\gamma\in
\Lip_\Eucl((0,1),\R^n)$
with $\g(0)=x $}
\\&
   \text{
$\g(1)= y$ and $\dot\gamma(t)\in
\big\{\textstyle\sum_{1\le j\le m} c_j X_{j, \gamma(t)}: \abs{c}\le
r\big\}$
for a.e. $t\in [0,1]$} \big\}.
\end{aligned}
\end{equation}
As usual, we call    \emph{Carnot--Carath\'eodory}  or
\emph{control} distance the
distance $\dcc$.   Note   that in the definition of $\dcc$ we may
choose paths
 $\gamma $ such that    $\dot\gamma =
\sum_j b_j(t) X_j(\gamma)$
where $b : (0,1)\to B_\Eucl(0,r)$ is \emph{measurable},
see Remark \ref{misurella}. In the present paper we shall make a prevalent use
of the distance $d$.
In the definition of both distances we agree that $d(x,y)=+\infty$
 if there are no paths in the pertinent class  which connect $x$ and $y$.

\paragraph{Horizontal regularity classes.}
Here we define   our notion
of horizontal regularity in terms of the distance $d$. Note that we  \emph{do
not} use the control distance $\dcc$.
\begin{definition}\label{montgomery}
  Let    $\H :=\{  X_1, ,\dots,X_m \}$ be a family of
  vector fields, $X_j\in C^1_\Eucl$. Let $d$ be their   distance \eqref{dicocco}
Let $g:\R^n\to \R$. We say that $g$ is \emph{$d$-continuous}, and we write
$g\in
C^0_\H(\R^n)$, if for all $x\in \R^n$, we have $
|g(y)- g(x)|\to 0$, as
$d(x,y)\to 0$.
We say that $g:\R^n\to \R$   is \emph{$\H$-Lipschitz} or \emph{$d$-Lipschitz}
in $A\subset\R^n$ if
\begin{equation*}
 \Lip_\H(g; A):= \sup_{ {x,y\in A,\; x\neq y}} \frac{|g(x)-
g(y)|}{d(x,y)}<\infty.
\end{equation*}
We say  that  $g\in C_\H^1(\R^n)$ if the derivative
$X_j^\sharp g(x): = \lim_{t\to 0}(f(e^{tX_j}x) - f(x))/t$ is a $d$-continuous
function
for any $j=1, \dots, m$.
 We
say that $g\in C^k_{\cal{H}}(\R^n)$ if all the derivatives
$X_{j_1}^\sharp\dots X_{j_p}^\sharp g$ are $d$-continuous for $p\le k$ and $j_1,
\dots,
j_p\in\{1,\dots,m\}$. If all the derivatives
$X_{j_1}^\sharp\dots X_{j_k}^\sharp g$ are $d$-Lipschitz on each $\Omega$
bounded set in the
Euclidean metric, then we say that $g\in
C^{k,1}_{\cal{H},\loc}(\R^n)$.
Finally, denote the usual Euclidean Lipschitz constant of $g$  on $A\subset\R^n$
by $\Lip_{\Eucl}(g; A)$.
\end{definition}

 We will  usually  deal with vector fields which are of class at least
$C^1_\Eucl\cap C^{s-1,1}_{\H,\loc}$, where $s\ge 1$ is a suitable integer.
In this case it   turns out that  commutators up to the order $s$ can be
defined, see Definition \ref{ffff} and Remark \ref{benposto}.
It will take a quite hard work (the whole Section \ref{identitone}) to show that
the different notions given in Definition \ref{ffff} actually agree.

\paragraph{Definitions   of commutator.}
Our  purpose  now  is to show that, given a family $\H$ of vector fields with
$X_j\in C^{s-1,1}_{\H,\loc} \cap C^1_\Eucl$, then commutators  can be   defined
up to
length $s$.

For any $\ell\in \N$,
denote by $\W_\ell:  = \{ w_1\cdots w_\ell: w_j\in \{1,\dots, m\}\}$ the words
of
length $\abs{w}:=\ell$ in the alphabet $1,2,\dots,m$. Let also $\perm_\ell $ be
the
group of  permutations of $\ell$ letters.
\begin{definition}[Coefficients $\pi_\ell(\s)$]
Define $\pi_\ell:\perm_\ell\to \{-1,0,1\}$ as follows: let
us agree that $\pi_1(\s)= 1$ for the unique $\s\in \perm_1$. Then,
for $\s\in \perm_2$, let $ \pi_2(\s): = 1$, if $\s(01)= 01$
 and
$\pi_2(\s) := -1 $, if $\s(01) = 10$.
Then, define inductively   for
$\ell\ge 2$,
\begin{equation} \label{16}
 \begin{cases}
 \pi_{\ell+1}(\wt \s ): = \pi_\ell(\s) &
\text{if\quad$\wt\s(01\cdots\ell)=0\s(1\cdots\ell)$  and $\s\in \perm_\ell$}
\\ \pi_{\ell+1}(\wt \s) := - \pi_\ell(\s )&
\text{if\quad$\wt\s(01\cdots\ell)=\s(1\cdots\ell)0$
and $\s\in \perm_\ell$}
\\ \pi_{\ell+1}(\wt\s): =0 &\text{if\quad$\wt\s_0(01\cdots\ell)\neq 0
\neq\wt\s_\ell(01\cdots\ell)$.}
\end{cases}\end{equation}
\end{definition}
Here we used the notation $\wt \s(01\cdots\ell) = \wt\s_0 (01\cdots\ell)
\wt\s_1 (01\cdots\ell)\cdots\wt\s_\ell (01\cdots\ell)$. The
coefficients $\pi_\ell$ are designed in order to write commutators in a
convenient way. Indeed, if   $A_1, \dots,
A_m:V\to V$ are  linear operators on a  vector space $V$, then one can check
inductively that, given a word $w= w_1\cdots w_\ell$, we have
\begin{equation}
\label{commototo}
[A_{w_1},[A_{w_2},\dots
[A_{w_{\ell-1}},A_{w_\ell}]]\dots ]=  \sum_{\s\in
\perm_\ell}\pi_\ell(\s) A_{\s_1(w)} A_{\s_2(w)}\cdots A_{\s_\ell(w)}.
\end{equation}
We will benefit later of the property
\begin{equation}\label{tacin}
 \pi_\ell(\s_1\cdots\s_\ell)=(-1)^{\ell+1}
\pi_\ell(\s_\ell\cdots\s_1)\quad\text{for all $\s\in\perm_{\ell}$.}
\end{equation}

We are now ready to define commutators for vector fields in our regularity
classes.
\begin{definition}
[Definitions of commutator]\label{ffff}
 Given a family  $\H=\{X_1,\dots X_m\}$ of vector fields  of class
$C^{s-1,1}_{\H,\loc}\cap C^1_\Eucl $,   define,  for $\psi\in C^1_\H$, \(X_j^\sharp\psi(x) :=
\cal{L}_{X_j}\psi(x),\) the Lie derivative; let also $
X_j\psi(x) := f_j(x) \cdot \nabla \psi(x)$ where $\psi\in C^1_\Eucl$.
Moreover, let
\begin{equation*}
\begin{aligned}
f_w &: =  \sum_{\s\in \perm_\ell } {\pi}_\ell(\s)\big( X_{\s_1(w)}\cdots
X_{\s_{\ell -1}(w)}
f_{\s_\ell(w)} \big)\quad\text{for all $w$ with $\abs{w}\le s$,}
 \\
X_w \psi  & :=[X_{w_1}, , \dots, [X_{w_{\ell-1}}, X_{w_\ell}]] \psi
: = f_w \cdot\nabla \psi\quad\text{for all $\psi\in C^1_\Eucl\quad \abs{w}\le
s$,}
\\
 X_w^\sharp      \psi&
: = \sum_{\s\in \perm_\ell } {\pi}_\ell(\s)
X_{\s_1(w)}^\sharp\cdots X_{\s_{\ell-1}(w)}^\sharp
X_{\s_\ell(w)}^\sharp\psi \quad\text{for all
$\psi\in C^{\ell}_{\cal{H}}$\quad $\abs{w}\le s-1$.
}
\end{aligned}
\end{equation*}
Given words  $u $ and $v $, define the (possibly non-nested) commutators
\begin{equation*}
 \begin{aligned}
f_{[u]v} & := X_u^\sharp f_v - X_v^\sharp f_u
\\&=
\sum_{ \alpha \in\perm_p , \beta\in \perm_q}
 \pi_p(\a)\pi_q(\beta)
\big(X_{\alpha_1(u)}\cdots X_{\alpha_p(u)}X_{\b_1(v)}\cdots X_{\b_{q-1}(v)}
f_{\b_q(v)}
\\&\qquad \qquad\qquad\qquad
-X_{\b_1(v)}\cdots   X_{\b_q(v)}X_{\alpha_1(u)}\cdots
X_{\alpha_{p-1}(u)}f_{\alpha_p(u)}\big),
\\
X_{[u]v} & := [X_u, X_v]
 := f_{[u]v}\cdot\nabla
= (X_u^\sharp f_v - X_v^\sharp f_u)\cdot\nabla \qquad\text{if
$\abs{u}+\abs{v}\le s$,}
\\  [X_u, X_v]^\sharp & :=X_{[u]v}^\sharp :=X_u^\sharp X_v^\sharp - X_v^\sharp
X_u^\sharp
\qquad\text{if
$\abs{u}+\abs{v}\le s-1$},
\end{aligned}
\end{equation*}
where $X_{[u]v}$ and $X_{[u]v}^\sharp$ act respectively on
$C^1_\Eucl$
and
$C^{\abs{u}+\abs{v}}_{\cal{H}}$ functions.
Finally,   for any $j\in\{1,\dots, m\}$ and
$w$ with $1\le \abs{w}\le s$, let
\begin{equation}\label{addio}
 \ad_{X_j} X_w \psi : = (X_j^\sharp  f_w - f_w\cdot\nabla f_j )\cdot\nabla \psi
=(X_j^\sharp f_w - X_w f_j) \cdot\nabla\psi \quad\text{for all  $\psi\in
C^1_\Eucl$.}
\end{equation}
\end{definition}
Note that we will never need in this paper
the commutators $X_w^\sharp\psi$ for $\abs{w}=s$.

\begin{remark} Let $Z\in \pm \H $, where $\H$ is a family in $C^{s-1,1}_{\H, \loc}
\cap C^1_\Eucl$.  If $\abs{w}\le s-1$, then there are no
problems in defining $\ad_Z X_w$. More precisely, in    Theorem \ref{dass} we
will see that  $\ad_Z X_w = [Z, X_w]$. If instead
$\abs{w}=s$, then
 the function
$t\mapsto
f_w(e^{t Z}x)$ is   Lipschitz continuous. In particular it is differentiable for
a.e.~$t$. In other words, for any fixed $x\in \R^n$, the limit
$
 \frac{d}{dt} f_w(e^{t Z}x) = : Z^\sharp f_w(e^{tZ}x) $
exists for a.e.~$t$ close to $0$.  Therefore the pointwise derivative $Z^\sharp
f_w(y)$ exists for almost
all
$y\in \R^n$ and ultimately $\ad_Z X_w$ is defined almost everywhere. See  the discussion 
in  Theorem \ref{seminuovo}-(b)  and Proposition \ref{arcano}.
\end{remark}
We will recognize that the first order operator $X_w$ agrees with
$X_w^\sharp$ against functions   $\psi\in C^{s-1,1}_{\cal{H},\loc} \cap
C^1_\Eucl$
and for $|w|\le s-1$. This is trivial if $|w|=1$, because   $X_k\psi :=
f_k\cdot\nabla\psi$ and $X_k^\sharp \psi := \cal{L}_{X_k}\psi$ are the same, if both
$X_k$ and
$\psi$ are $C^1_{\Eucl}$.

\begin{remark}\label{benposto} Both our definitions of commutator, $X_w$ and
$X_w^\sharp$
are well
posed from an
algebraic point of view. Indeed, it is
easy to check that $[X_u, X_v] =
(X_u^\sharp f_v - X_v^\sharp f_u)\cdot\nabla= -[X_v,
X_u]$. Moreover
\begin{equation*}
\begin{aligned}
[X_w, [X_u, X_v]] & = (X_w^\sharp f_{[u]v} - [X_u, X_v]^\sharp f_w)\cdot\nabla
\\&= \{ X_w^\sharp X_u^\sharp f_v - X_w^\sharp X_v^\sharp f_u - X_u^\sharp
X_v^\sharp f_w +
X_v^\sharp X_u^\sharp f_w\}\cdot\nabla,
\end{aligned}
\end{equation*} for any $u,v,w$ with $|u|+|v|+|w|\le s$.
This immediatley implies the Jacobi identity
\begin{equation}\label{jacopa}
 [X_u,  [X_v,X_w ] ] + [X_{v},[X_w. X_u]] + [X_w,[X_u,X_v]]=0.
                                           \end{equation}
 Antisymmetry and the  Jacobi identity for the
  commutators
$X_w^\sharp$ can be checked with the same argument.
\end{remark}

Let $\Omega_0\subset\R^n$ be a fixed open set, bounded in the Euclidean metric.
Given a family   $\H$  of vector fields of
 class $C^1_\Eucl\cap C^{s-1,1}_{\H,\loc}$,  introduce the
constant
\begin{equation}
\label{lipo}
\begin{aligned}
L_0  :&  = \sum_{j_1,\dots, j_s=1}^m\Big\{ \sup_{\Omega_0 }
\Big(  |f_{j_1}| +
|\nabla
f_{j_1}|
+
  \sum_{p\le s} |X_{j_1}^\sharp \cdots
 X_{j_{p-1}}^\sharp f_{j_p}|\Big)
\\  & \qquad
\qquad \qquad +\Lip_\H(X_{j_1}^\sharp\cdots
X_{j_{s-1}}^\sharp f_{j_s};   \Omega_0)\Big\}.
  \end{aligned}
       \end{equation}
Fix also $\Omega\Subset\Omega_0$. We shall always choose points  $x\in\Omega$ and we fix a
constant
$t_0>0$ small enough
to ensure that
\begin{equation}\label{hello}
e^{\t_1 Z_1}\cdots e^{\t_N Z_N}x\in \Omega_0
\quad \text{if $x\in \Omega$,
$Z_j\in \H$, $\abs{\t_j}\le t_0$ and $N\le N_0  $,}
\end{equation}
 where
$N_0$ is a suitable  algebraic constant which depends  on the    data  $n, m$ and
$s$ associated with the family $\H$.

 \subsection{Non commutative   formulas}
\label{semivecchio}
In this section we discuss some preliminary tools on noncommutative
calculus.  Some of the objects we discuss here already
appeared in \cite{MM}, for H\"ormander vector fields,  in a   higher
regularity setting. Observe that Theorem \ref{seminuovo} has also a relevant role in \cite[Lemma~3.1 and Theorem~3.5]{MontanariMorbidelli11d}.

\begin{theorem}\label{seminuovo}
 Let $\H$ be a family of $C^1_\Eucl\cap C^{s-1,1}_{\H,\loc}$ smooth vector
fields. Fix $Z\in\pm\H$ and $X_w $ with  $|w|\ge 1$.
    Then:
 \begin{itemize*}
  \item [(a)] if $|w|\le s-1$,   then, for all $\psi\in C^1_\Eucl$, $y\in
\Omega$ and  $|t|\le t_0$ (see \eqref{hello}) we
have
\begin{equation}\label{lachi}
 \frac{d}{d
 t}X_w(\psi  e^{-t Z})(e^{t Z}y)=\ad_{Z} X_w(\psi e^{-t Z})(e^{t Z}y);
\end{equation}
  \item [(b)]
 if   $|w|=s$,   then  for any  $\psi\in C^1_\Eucl$ and $y\in \Omega$
the function $\phi(t):=  X_w (\psi e^{-t Z})(e^{t Z}y) $ is Euclidean
Lipschitz and satisfies
\begin{equation}\label{arte2}
 \frac{d}{dt}X_w (\psi e^{-tZ})(e^{t Z}x) = \ad_{Z} X_w(\psi e^{-t
Z})(e^{t Z}x)\quad\text{for a.e. $t\in (-t_0, t_0)$}.
\end{equation}
\end{itemize*}
\end{theorem}

To comment on   \eqref{lachi},  assume that $Z= X_j$ for some $j\in
\{1,\dots,m\}$. Note that  in Theorem
\ref{dass} we will show that
$\ad_{X_j} X_w=[X_j , X_w] = X_{jw}$, if $|w|\le s-1$, $j\in\{1, \dots, m\}$.
Looking instead at   equation \eqref{arte2},  the operator
 $\ad_{Z}X_w$ has been
defined in \eqref{addio}. If we assume the H\"ormander condition of step $s$, we shall see in
 Proposition~\ref{arcano} that we can write
\begin{equation}
 \label{campiello}
 \cal{L}_Z X_w ( \psi e^{-tZ})(e^{tZ}x):= \frac{d}{dt}X_w(\psi e^{-tZ})(e^{t
Z}x)= \sum_{1\le |u|\le s} b^u( t
 ) X_u  (\psi e^{-tZ})(e^{t Z}x),
\end{equation}
where the functions $b^u $  may depend on $Z, w,x$ and  are measurable   and
bounded.

\begin{remark}\label{farlocco}
The proof of \eqref{lachi} is standard for smooth (say at least $C^2$) vector
fields, see \cite[Proposition 1.9]{KobayashiNomizu},
or, for a different argument,       \cite[Proposition
3.5]{Aubin01}  and \cite[Lemma 3.1]{MM}.
Note also that
 $U e^{-tV}(e^{tV}x) = e^{-tV}_*(U_{e^{tV}x}) $, by
definition of tangent map. Thus, $
 (\mathcal{L}_V U)_x
=\frac{d}{dt} U
e^{-tV}(e^{tV}x)\bigr|_{t=0}$, by definition of Lie derivative.
Then, in \textit{Step~1}  of the proof below, we are proving nothing
but the
probably known fact that
$\mathcal{L}_VU = (V\xi - U\eta )\cdot\nabla$, if $V=\eta\cdot\nabla$ and $U=
\xi\cdot\nabla\in C^1_\Eucl$.
\end{remark}

\begin{proof} [Proof of Theorem \ref{seminuovo}]
We split the proof in
three steps.

\step {Step 1.}   We prove     that for any   $ U =
\xi \cdot\nabla
\in C^1_\Eucl$ and $V= \eta \cdot\nabla\in
C^1_\Eucl$, we have for all $x\in \R^n$
\begin{equation*}
 \frac{d}{dt} U (\psi e^{-tV })(e^{tV }x) = [V, U](\psi e^{-tV})(e^{tV }x)
\quad\text{if    $|t|$ is small enough.}
\end{equation*}
Here $[V, U]:= (V\xi - U\eta  )\cdot\nabla$ and $\psi\in C^1_{\Eucl}$.

Let $\psi\in C^1_\Eucl$. Take the usual smooth approximations
  $V^{{\s}} = \eta^\s\cdot\nabla $,  $U^{{\s}} = \xi^\s\cdot\nabla$ and
$\psi^\s$.
Since $V$ and $U$ are $C^1$, elementary  properties of Euclidean mollifiers show
that
$\eta^{{\s}}\to \eta$, $\xi^{{\s}}\to \xi$   and
$V^{{\s}} \xi^\s -
U^\s \eta^\s \to V\xi-U\eta$,  uniformly
on compact sets,
as $\sigma\to 0$.
Therefore, we have
\begin{equation*}\begin{aligned}
  \frac{d}{dt} U^{{\s}} (\psi^\s e^{-tV^{{\s}}})(e^{tV^{{\s}}}x)
&= V^{{\s}}
U^{{\s}}(\psi^\s
  e^{-tV^{{\s}}})(e^{tV^{{\s}}}x)
-U^{{\s}} V^{{\s}}(\psi^\s
  e^{-tV^{{\s}}})(e^{tV^{{\s}}}x)
\\&
=(V^\s \xi^\s - U^\s \eta^\s)(e^{t V^\s}x) \cdot \nabla(\psi^\s e^{-tV^\s})(e^{t
V^\s}x)=:R(\s).
\end{aligned}
\end{equation*}
First equality is provided in textbooks, see Remark \ref{farlocco}.
 In our notation, the intermediate term here is  $[V^\s,
U^\s]^\sharp(\psi e^{-tV^{{\s}}})(e^{tV^{{\s}}}x)$. Both its addends
  may have a not clear behaviour, as $\s\to 0$. After the cancellation,
  second derivatives
against $\psi^\s e^{-t V^\s}$ disappear and we may let $\s\to 0$ in the second
line.
By standard ODE
theory, see for example \cite[Chapter 5]{Hartman},
 $\nabla e^{t
V^{{\s}}}\to \nabla e^{tV}$, uniformly on $\Omega$ and $|t|\le t_0$,
as $\sigma\to 0$. Then $\lim_{\s\to 0}R(\s)
= [V, U ](\psi e^{-tV })(e^{tV }x),$ uniformly on $t\in [-t_0, t_0]$
and $x\in \Omega$.
 Moreover,
\[ \lim_{\s\to 0}
U^\s (\psi^\s e^{-tV^\s})(e^{tV^\s}x)
=U  (\psi  e^{-tV} )(e^{tV }x) \quad \text{for all } t\in [-t_0, t_0]\quad  x\in
\Omega.
\]
Therefore,  Step 1 is accomplished.

\step{Step 2.}  We prove statement    (a).  By
uniqueness of the flow of
 ${Z}$, we may work with  $t=0$.
  \begin{equation*}
\begin{aligned}
 \frac{d}{dt} X_w(\psi e^{-t{Z}})(e^{t{Z}}x) \Big|_{t=0}
 =\lim_{t\to 0}\frac{1}{t}
 \Big\{ &  [f_w(e^{t{Z} }x)- f_w(x)]\cdot
 \nabla (\psi e^{-t{Z}})(e^{t{Z}}x)
\\ \qquad
  & + f_w (x) \cdot [\nabla(\psi e^{-t{Z}})(e^{t{Z}}x) - \nabla  \psi
(x)]\Big\}.
\end{aligned}
\end{equation*}
But $\lim_{t\to 0}\frac 1 t  [f_w (e^{t{Z} }x)- f_w (x)]= :{Z}^\sharp f_w (x)$
exists, 
because $ f_w\in C^1_{\cal{H}}$.
Moreover, since $\p_1,\dots, \p_n, {Z}\in C^1_\Eucl$,
Step 1
gives for all $\alpha\in\{1,\dots,n\}$,
\begin{equation}\label{snooze}\begin{aligned}
 \lim_{t\to 0}
\frac{1}{t}  & [\p_\alpha (\psi e^{-t{Z}})(e^{t{Z}}x) - \p_\alpha
\psi (x)]
 =[Z, \p_\a]\psi(x) = -\p_\a f(x)\cdot\nabla\psi(x),
\end{aligned}
\end{equation}
where $Z= f\cdot\nabla$.
This   concludes the proof of Step 2.

\step{Step 3.} Proof of   (b).  We will show that $t\mapsto X_w(\psi e^{-t
Z})(e^{tZ}x)=:\phi(t)$ is
Lipschitz continuous on
$[-t_0, t_0]$ for all $x\in \Omega$.
\begin{equation*}
\begin{aligned}
 \phi(\t)- \phi(t)&= f_w(e^{\t Z}x)\cdot \big[
\nabla (\psi e^{-\t Z})(e^{\t Z}x) - \nabla (\psi e^{-t Z})(e^{t
Z}x)\big]
\\&\quad + [f_w(e^{\t Z}x)- f_w(e^{t Z}x)]\cdot \nabla  (\psi e^{-t
Z})(e^{t Z}x).
\end{aligned}
\end{equation*}
But, since $f_w\in \Lip_{\cal{H}}$, we have
$|f_w(e^{\t Z}x)- f_w(e^{t Z}x)|\le C|\t- t|$. Moreover,
if  $t\in (-t_0, t_0)$
 is a differentiability point for $t\mapsto f_w(e^{tZ }x)$, we have
$
\lim_{\t\to t} (f_w(e^{\t Z}x)- f_w(e^{t Z}x))/(\t-t) = Z^\sharp
f_w(e^{tZ}x),
$ 
by definition of derivative along $Z$.
Finally, for any $\alpha\in\{1, \dots, n\}$, \eqref{snooze} shows that the function 
$t\mapsto \p_\a(\psi e^{-t Z})(e^{tZ}x)\in C^1_\Eucl$ and that
\begin{equation*}
 \frac{d}{dt}\p_\a  (\psi e^{- t Z})(e^{t Z}x) = -\p_\a
f(e^{t Z}x) \cdot\nabla (\psi e^{-t Z})(e^{t Z}x).
\end{equation*}
Then the proof of (b) is easily concluded. 
\end{proof}

\subsection{Integral remainders}
Here we introduce a class of integral remainders $O_p(t^\lambda, \psi, y)$.
There is a  reason why we use a different notation from  the seemingly
similar remainders $R_p(t^\lambda,\psi,y )$ appearing in the Taylor formula
below, see
\eqref{rospo}. Namely,   the remainders $O_{p}(\dots)$ have a
much more
``balanced'' structure. Under suitable involutivity conditions and using such balanced structure, in 
\cite{MontanariMorbidelli11d} 
 we will be able to show
that they can be
given a \emph{pointwise} form (a  consequence  of this fact is  the pointwise form of the  remainders in   expansion~\eqref{polacco}).

Let $\H$ be a family in the regularity class
$C^{s-1,1}_\H\cap C^1_\Eucl$. Let $\lambda\in\N$,
$p\in\{2,\dots,s+1\}$. We
denote, for  $y\in \Omega$   and    $t\in[0,t_0]$, $\psi\in C^1_\Eucl(\Omega_0)$
\begin{equation}\label{nota}
 O_{p }(t^\lambda,\psi ,y)= \sum_{i=1}^N\int_0^t
 \o_i(t,\t )\frac{d}{d\t} X_{w_i}(\psi \phi_i^{-1} e^{-\t Z_{i}})( e^{\t
Z_{i}}\phi_i y)d\t,
\end{equation}
where $N$ is a suitable integer and  $\psi $ is the identity map or
$\psi =\exp(t Z_1)\cdots \exp(t Z_\mu),$ for some  integer $\mu$ and suitable
vector
 fields $Z_j\in\pm \H$.
The ``balanced
structure'' we mentioned above,  follows from identity $(\psi
\phi_i^{-1} e^{-\t
Z_{i}})( e^{\t
Z_{i}}\phi_i y)=\psi(y)$.

To describe the generic  term of the sum  above, we drop the dependence on $i$:
\begin{equation}\label{tipicissimo}
(*):= \int_0^t
 \omega(t,\t) \frac{d}{d\t} X_w(\psi \phi^{-1}e^{-\t  X})( e^{\t X}\phi y)d\t .
\end{equation}
  Here $X_w$ is a commutator of  length   $|w|=p-1$ and  $X\in  \pm \H$.
Moreover, for
 any $t<t_0$,
 the function  $\o(t,\t)$ is a
 polynomial, homogeneous  of degree $\lambda -1$ in  all variables $(t,\t)$,
 so that
\begin{equation}
 \label{opp}
 \int_0^t \omega (t,\t ) d\t =  b t^\lambda\quad \text{for any }  \;t>0
\end{equation}
for a suitable constant $b\in\R$.
  The map $\phi$  is the identity  map or otherwise it      has the form
$\phi= \exp(t  Z_1)\cdots \exp(t  Z_\nu)$
for some $\nu\in\N$, where    $Z_j\in\pm\H$.
 Observe that, if $p\le s$, i.e. $\abs{w}\le s-1$,  Lemma \ref{seminuovo}
(a) gives
 \[
 (*)= \int_0^t
 \omega(t,\t)  \ad_X X_w(\psi \phi^{-1}e^{-\t  X})( e^{\t X}\phi y)d\t .
\]
Therefore the remainder has the same form of the analogous term in
\cite[Eq.~(3.5)]{MM},
provided that  we are able to show that $\ad_X X_w = [X, X_w]$ (this will be
achieved in Theorem~\ref{dass}).
If instead $p= s+1$, i.e. $\abs{w}= s$, we need to use part  Theorem
\ref{seminuovo}-(b) to get some information on the remainder. See also Proposition 
\ref{arcano} and see the paper \cite{MontanariMorbidelli11d} 
for a detailed discussion of  remainders of higher order
$O_{s+1}(\cdots)$.

   A remainder of the form \eqref{nota} satisfies for every
$\alpha,\lambda\in\N$ and $p\le s+1$   estimate
\begin{equation}\label{laprop}
 t^\a O_p{(t^\lambda,\psi ,y})= O_p({t^{\a+\lambda},\psi ,y}) \quad \text{for
all}\quad  y\in \Omega \quad t\in[0, t_0].
\end{equation}
Let us also recall estimate \begin{equation}\label{tredieci}
 | O_{p }(t^\lambda,\psi ,y)|
\le C t^\lambda,                            \end{equation}
which holds for $p\le s+1$, $\lambda\in \N$.
To see~\eqref{tredieci}, just observe that, at any $t$
and
for $j\in\{1,\dots,m\}$ and $\abs{w}\le s$, we have
\begin{equation*}
\begin{aligned}
\Bigl| \frac{d}{d\t} X_w (\psi\phi^{-1}e^{-\t X_j})(e^{\t X_j}\phi x)\Bigr| &=\bigl| \ad_{X_j}X_w
(\psi\phi^{-1}e^{-\t X_j})(e^{\t X_j}\phi x)\bigr|\le C,
\end{aligned}
\end{equation*}
at any  $\t$ such that  $ X_j^\sharp f_w(e^{\t X_j}\phi x)$ exists.  Here we use the trivial estimate
  $\abs{\ad_{X_j}X_w}\le
\abs{X_j^\sharp f_w}+\abs{X_w f_j}\le C$, because $f_w\in\Lip_\H$ and $f_j\in C^1_\Eucl$.

Note finally that, if $j\in\{1, \dots, m\}$,  $p\le s+1$ and   $Z\in
\pm \H$, we have
\begin{equation*}\begin{aligned}
            O_p(t^\lambda, \psi e^{tZ}, y)= O_p(t^\lambda, \psi , e^{tZ} y).
                \end{aligned}
\end{equation*}

\begin{proposition} \label{lasposto}
Assume that
   $p\le s$ and assume that
$\ad_{X_j} X_w = X_{jw}$
for all word  $w$ with length $\abs{w}\le  p-1$ and $j\in\{1, \dots,
m\}$. Then there are constants
$c_w$, $|w|=p,$ such that
\begin{equation}\label{iii} \begin{aligned}
 O_{p }(t^\lambda,\psi ,y)
  &= \sum_{|w|=p}c_w t^{\lambda}X_w \psi (y) + O_{p +1}(t^{\lambda+1} ,\psi ,y).
\end{aligned}
\end{equation}
 \end{proposition}
The proof of Proposition \ref{lasposto} has been given in
\cite[Proposition 3.3]{MM} for smooth vector fields. The argument is the same
in our case. One just needs to check that all the  computations we made there work
perfectely in our regularity setting.
We omit the details.

\begin{remark}
\begin{itemize*}
\item The statement of Proposition \ref{lasposto}, and  in particular
the assumption $\ad_{X_j}X_w = X_{jw}$, is designed in order to be a part of the
induction machinery  we shall implement to prove Theorem
\ref{dass} in the
following section.
\item
The generalization of \eqref{iii} to the case $p= s+1$ is discussed under the H\"ormander condition in Section \ref{cinquemila}, see \eqref{arte}. 
In the companion paper \cite{MontanariMorbidelli11d} we deal with a more general situation.
\end{itemize*}
\end{remark}

\begin{remark}  
Let for a while $\H= \{X_1, \dots, X_m\}$ be a family of smooth vector fields.
Iterating Theorem \ref{seminuovo}, we have for $x\in\Omega$ and $\abs{t}$ 
sufficiently small
 \label{ittero}
\begin{equation}
\label{itero}
\begin{aligned}
X_w (\psi e^{-t Z_\mu}\cdots e^{-t Z_1})x & = \sum_{\abs{\a}=0}^\ell
\ad_{Z_\mu}^{\a_\mu}\cdots\ad_{Z_1}^{\a_1} X_w\psi(e^{-t Z_\mu}\cdots e^{-t
Z_1}x)\frac{t^{\abs{\a}}}{\a!}
\\&\qquad + O_{\ell+\abs{w}}(t^{\ell+1},\psi,x).    \end{aligned}
\end{equation}
Formula \eqref{itero} will
be referred to later.
\end{remark}

\paragraph{Taylor formula with integral remainder.}
Here we show that
functions of class $C^{s-1,1}_{\H,\loc}$ enjoy an elementary
Taylor expansion with integral remainder.
Let $p,\lambda\in \N$. Denote by
$R_p(t^\la, \psi,x)$ a sum of a  finite  number of terms of the
form
 \begin{equation}\label{rospo}
  \int_0^t \omega (t, \tau)\frac{d}{d\tau}(X_{j_1}^\sharp)^{k_1} \cdots
(X_{j_\mu}^\sharp)^{k_\mu} \psi
(e^{\tau X_i}\phi x)d\t,        \end{equation}
where the polynomial   $\o(t,\tau)$ is
  homogeneous  of degree $\lambda -1$ in
all variables $(t,\tau)$. 
 This ensures that
$ \int_0^t \omega (t,\tau ) d\tau =  C t^\lambda,$ for any $t>0$.
  Moreover,  $i ,  j_1, \dots, j_\mu \in \{1,\dots,m\}$,  $k_1 +\cdots + k_\mu =
p-1$.
  The map $\phi$  is the identity map or it      has the form
$\phi= \exp(t  Z_1)\cdots \exp(t  Z_\nu)$
for some  $\nu\in\N$, where
$Z_j\in\pm \H$. If $\Omega\subset\R^n$ is bounded, then we have,
for all $x\in \Omega$,  $\abs{t}\le t_0$,
\begin{equation*}
 |R_p(t^\la, \psi, x)|
\le C  \Lip_\H \big((X_{j_1}^\sharp)^{k_1}\cdots (X_{j_\mu}^\sharp)^{k_\mu}\psi,
B_d(x, C\abs{t})\big)
t^\la,
\end{equation*}
where $t_0$ is positive, small enough, see \eqref{hello}.

Denote $\D^{j_1\cdots j_q}x:= \D^{j_1}\D^{j_2}\cdots\D^{j_q}x:= e^{t
X_{j_1}}
\cdots e^{t X_{j_q}}x,$ where $j_1, \dots, j_q\in \{1,\dots, m\}$.
\begin{lemma} Let $\psi
\in C^{\ell-1, 1}_{\H,\loc}$, for some $\ell\le s$.  Then
for any $q\ge 1$ and $j_1, \dots, j_q\in\{1,\dots,m\}$, we  have in standard
multi-index notation
\begin{equation}
 \label{ttt}\psi(\D^{j_1\cdots j_q}x )= \sum_{\substack {k_1, \dots, k_q\geq
0 \\ k_1+\cdots + k_q \le \ell-1}}
(X_{j_q}^\sharp)^{k_q}\cdots
(X_{j_1}^\sharp)^{k_1}\psi(x)\frac{t^{\abs{k}}}{k!}+ R_\ell(t^\ell,
\psi,x).
\end{equation}
\end{lemma}

\begin{proof}
We prove formula \eqref{ttt} by induction on $q\ge 1$. Fix $j\in\{1, \dots,
m\}$.
Let $\psi\in C^{\ell-1, 1}_{\cal{H},\loc}$, Then
$ \bigl(\frac{d}{dt}\bigr)^k \psi(e^{tX_j}x) = (X_j^\sharp)^k
\psi(e^{tX_j}x),$ for
$k=0,1,\dots, \ell-1$.
Moreover, the function $t\mapsto (X_j^\sharp)^{\ell-1}  \psi(e^{tX_j}x)$ is Euclidean
Lipschitz and, for a.e.~$t\in(-t_0, t_0)$, its derivative can be estimated by
$\Lip_\H((X_{j}^\sharp)^{\ell-1} \psi; B_d(x, t_0))$. Therefore,
  the Taylor formula gives
\begin{equation*}
\begin{aligned}
 \psi(e^{tX_j}x)& = \sum_{k=0}^{\ell-1}(X_j^\sharp)^k \psi(x) \frac{t^k}{k!}
+\int_0^t\frac{(t-\tau)^{\ell-1}}{(\ell-1)!}\frac{d}{d\tau}(X_j^\sharp)^{\ell-1}
\psi
(e^{\tau
X_j}x) d\tau
\\&=\sum_{k=0}^{\ell-1}(X_j^\sharp)^k \psi(x) \frac{t^k}{k!} + R_\ell(t^\ell,
\psi,x).
\end{aligned}
\end{equation*}

Next we give the induction step.  Let $q\ge 1$. Then,
\begin{equation*}
 \begin{aligned}
  \psi(\D^{j_0}\D^{j_1\cdots j_q}x)
&=
  \sum_{k_0=0}^{\ell-1}(X_{j_0}^\sharp)^{k_0} \psi
(\D^{j_1\cdots j_q}x) \frac{t^{k_0}}{k_0!} + R_\ell(t^\ell, \psi, \D^{j_1\cdots
j_q}x)
\\&=
\sum_{k_0=0}^{\ell-1} \frac{t^{k_0}}{k_0!}
\Big\{\sum_{ \substack{k_1, \dots, k_q\ge 0\\ k_1+\cdots+ k_q\le \ell-1 -k_0}}
  (X_{j_q}^\sharp)^{k_q}\cdots
(X_{j_1}^\sharp)^{k_1}(X_{j_0}^\sharp)^{k_0}\psi(x) \frac{t^{k_1+\cdots +
k_q}}{k_1!\cdots k_q!}
\\&\qquad\qquad
+ R_{\ell- k_0} (t^{\ell- k_0}, (X_{j_0}^{\sharp})^{{k_0}}\psi,x )\Big\}+
R_\ell(t^\ell,
\psi,x).
\end{aligned}
\end{equation*}
The proof is concluded by   property $ t^{k_0} R_{\ell- k_0} (t^{\ell-
k_0},
(X_{j_0}^\sharp)^{k_0}\psi,x) = R_\ell(t^\ell, \psi,x)$.
\end{proof}

\section{Commutator identities}
\label{identitone}
In this section we show that, if the vector fields  of the family $\H$
belong to  $C^{s-1,1}_{\H,\loc}
\cap C^1_\Eucl$, then  the    various notions of commutators introduced
in
Definition \ref{ffff} agree. This requires a quite elaborate algebraic
work
which will be performed in the first subsection. Later on, 
we will show that
the   machinary we construct, in particular the generalized Jacobi
identities in  Proposition \ref{genero} can be   useful  to detect
nested commutators identities.

  Let $\H = \{X_1, \dots, X_m\}$ be
family
of vector fields of  class $C^1_\Eucl\cap C^{s-1,1}_{\H,\loc}$. We use the
notation
$\W_\ell$ to indicate the set of words $w= w_1\cdots w_\ell$ of length
$\ell$.

The main result of
this section
is the following theorem,
 which has a key role in the proof of 
\cite[Theorems~3.5 and~3.8]{MontanariMorbidelli11d}   and ultimately of Theorem \ref{teoremahormander} here;
see the discussion at the beginning of \cite[Subsection~3.2]{MontanariMorbidelli11d}.

\begin{theorem}\label{dass}
Let $\H$ be a family of vector fields of class
$C^1_\Eucl\cap C^{s-1,1}_{\H,\loc}$. Then,
 if $1\le \ell\le s-1$, the
following statements are equivalent and true.
\begin{itemize*}
 \item [(i)] For any $w\in \W_\ell  $ and for all
$\psi\in C^1_\Eucl\cap C^{\ell,1}_{\H,\loc}$ we have
\begin{equation}\label{dagh}
 X_w\psi = X_w^\sharp\psi.
\end{equation}
\item[(ii)]
For any $Z=\psi \cdot\nabla \in C^1_\Eucl\cap C^{\ell ,1}_{\H,\loc}$ and for all
$w\in
\W_\ell$,  we have
\begin{equation}\label{dasha}
 \ad_{Z} X_w \phi =[Z,X_{w}] \phi \qquad \text{for all  $\phi\in C^1_\Eucl$}.
\end{equation}
\end{itemize*}
\end{theorem}
\begin{remark}\label{lachi3}
In view of Theorem \ref{dass}, formula \eqref{lachi} in
 Theorem~\ref{seminuovo} becomes,
\begin{equation}\label{lachi2}
 \frac{d}{d
 t}X_w(\psi e^{-t Z})(e^{t Z}y)=[Z,X_w](\psi e^{-t Z})(e^{t Z}y)
\quad\text{if $|w|\le  s-1$\quad $t\in(-t_0,t_0)$}.
\end{equation}
The case $\abs{w}=s$ will be discussed in Section \ref{cinquemila}, see 
  e.g.~Proposition~\ref{arcano}.
\end{remark}

To prove  Theorem \ref{dass}, we need the following proposition which may have
some independent interest.

\begin{proposition}[Generalized Jacobi identities]
\label{genero}   Let $\H$ be a family in the regularity class
$C^{s-1,1}_{\H,\loc}
\cap C^1_{\Eucl}$.
For any $v\in \W_p$, $w\in \W_q$, $p,q\ge 1$, $p+q\le s$, we
have
\begin{equation}\label{otto}
 X_{[ v]w}= \sum_{\s\in \perm_p}\pi_p(\s) X_{\s_1(v)\dots\s_p(v)w}.
\end{equation}
If $\abs{w}=0$, then \eqref{otto} fails, but  for any $v= v_1\cdots
v_\ell\in \W_\ell$,
$\ell\le s$, we have
\begin{equation}\label{J2}
 X_{v} =\frac 1\ell \sum_{\s \in \perm_\ell}
\pi_\ell(\s)X_{\s_1(v)\dots\s_\ell(v)}.
\end{equation}
\end{proposition}
Before proving the proposition, to explain the reason of our terminology,  we
give a couple of examples to show that our identities, suitably specialized,
give back some familiar identities. See also Subsection \ref{chetbaker}.
\begin{example} Let
$X_1, X_2$ and $X_3$  be sufficiently  smooth vector fields. Then
\begin{equation*}
\begin{aligned}
 [X_1, [X_2, X_3]] & = : X_{123} =
\frac{1}{3}\sum_{\s\in\perm_3}\pi_3(\s)X_{\s_1(123)\s_2(123)\s_3(123)}
\\&= \frac 13\big\{X_{123} - X_{132} - X_{231} + X_{321}\big\}
\\&= \frac 13 \big\{[X_1, [X_2, X_3]] - [X_1, [X_3, X_2]] - [X_2,[X_3,
X_1]]+ [X_3,[X_2, X_1]]\big\}
\\&= \frac 23 [X_1, [X_2, X_3]] -
\frac 13  [X_2,[X_3,X_1]] -\frac 13 [X_3, [X_1, X_2]].
\end{aligned}
\end{equation*}
Comparing the first and the list line one can recognize the familiar  Jacobi
identity.
\end{example}
\begin{example}
Here, looking at the fourth order  identity \eqref{J2} with $\ell=4$ and taking
$w= 1212$,  we check the nested commutators identity
\begin{equation}\label{oteo}
 X_{1212} = X_{2112} = -X_{1221}
\end{equation}
discussed in \cite[eq.~(4.3)]{Oteo}.
To get   \eqref{oteo}, start from the $4$-th order formula
\begin{align*}
 X_{1234} &= \frac{1}{4}\{X_{1234}- X_{1243}- X_{1342}+ X_{1432} - X_{2341}+
X_{2431} + X_{3421} - X_{4321}\}.
\end{align*}
Letting  $1$ instead of  $3$ and $2$ instead of $4$, we get
\begin{align*}
4 X_{1212}   &=X_{1212}  - X_{1221}- X_{1122}+ X_{1212} - X_{2121}+
X_{2211}+ X_{1221}- X_{2121} ,
\end{align*}
which is equivalent to $ 2 X_{1212} = - 2X_{2121}$, and gives immediately
\eqref{oteo}.
\end{example}

\begin{proof}[Proof of Proposition \ref{genero}] To prove \eqref{otto}, we
argue by  induction on $\abs{v}$.
The property is trivial if
$\abs{v}=1$ and $1\le \abs{w}\le s-1$. Assume
that for a given $p\in\{1,\dots,s-2\}$, formula \eqref{otto} holds for all $v,w$
with $\abs{v}=p$ and $1\le \abs{w}\le s-p$
and we will prove that it holds for any $v,w$
with $|v|= p+1$ and $1\le \abs{w}\le s-p-1$.

Write  $\wt v= v_0 v\in \W_{p+1}$ and $\wt\s(\wt v) = \wt\s_0(v_0 v)\dots
\wt\s_p(v_0 v)$. Then the defining property \eqref{16} of the coefficients $\pi(\s)$ gives
 \begin{equation*}
\begin{aligned}
&\sum_{\wt \s\in \perm_{p+1}} \pi_{p+1}(\wt \s)  X_{\wt\s_0(\wt v)\wt\s_1(\wt
v)\cdots\wt\s_p(\wt v)w}
\\&=\sum_{\s\in\perm_p}\pi_p(\s)\big(
X_{ v_0 \s_1(  v)\cdots \s_p(  v)w }-
X_{  \s_1(  v)\cdots \s_p(  v) v_0w} \big)
\\&
 = \Big[X_{v_0}, \sum_{\s\in\perm_p}\pi_p(\s) X_{\s_1(v)\cdots\s_p(v)w}
\Big]-\sum_{\s\in\perm_p}\pi_p(\s)X_{\s_1(v)\cdots\s_p(v) v_0 w}
\quad \text{(inductive assumption)}
\\&=[X_{v_0}, X_{[v]w}] - X_{[v]v_0w}
= X_{[v_0 v]w},
\end{aligned}
 \end{equation*}
by the Jacobi identity \eqref{jacopa} and the antisymmetry. Thus \eqref{otto}
is proved.

To prove   \eqref{J2}, we work by induction. The statement for $\ell=2$ is
obvious.
 Assume that \eqref{J2} holds for
some $\ell\in\{2,\dots, s-1\}$. We need to show that
\begin{equation*}
 X_{v_0v}= \frac{1}{\ell +1}\sum_{\s\in \perm_{\ell+1}}
\pi_{\ell+1}(\wt\s) X_{\wt\s_0\wt\s_1\dots\wt\s_\ell}
\quad\text{for all $v_0v= v_0 v_1\cdots v_\ell\in \W_{\ell+1}$,}
\end{equation*}
where for all $j$  we denoted $\wt \s_j= \wt\s_j(v_0v)$.
But the  definition of $\pi_{\ell+1}$,  the induction assumption and
\eqref{otto} show
that
\begin{equation*}
 \begin{aligned}
 \sum_{\wt \s\in \perm_{\ell+1}}
  \pi(\wt \s ) X_{\wt \s_0\wt \s_1 \cdots \wt \s_\ell} &=
  \sum_{\s\in \perm_\ell}  \pi_\ell(\s) (X_{v_0\s_1\cdots\s_\ell} -
X_{\s_1\cdots\s_\ell v_0})
 \\&=[X_{v_0},\sum_{\s\in\perm_\ell}\pi_\ell(\s)X_{\s_1\dots\s_\ell}]-
\sum_{\s\in\perm_\ell}
\pi_\ell (\s)
X_{\s_1\cdots\s_\ell v_0}
\\&=  \ell X_{v_0v}  -X_{[v]v_0} =(\ell+1) X_{v_0v},
 \end{aligned}
\end{equation*}
as desired.
\end{proof}

Recall the notation $\Delta^{k_1\cdots k_\ell }x:= e^{t X_{k_1}}\cdots e^{t
X_{k_\ell}}x$, where $\ell\in \N$ and $k_j\in\{1,\dots,m\}$.
\begin{lemma}\label{lemmauno}
For any $\ell\in\{2,\dots, s-1\}$, for each
$w\in\W_\ell$ and for each  $\psi\in   C^{\ell,1 }_{\H,\loc}$, we have
\begin{equation*}
X_w^\sharp  \psi(x)
= \lim_{t\to 0}
\frac{1}{t^\ell}
\sum_{\s\in \perm_\ell}\pi_\ell(\sigma)
\psi(\D^{\s_\ell(w) \cdots\s_1(w)}x)\quad\text{for all $x\in\R^n$}.\end{equation*}
\end{lemma}

\begin{proof}  To prove the statement, we shall show  the Taylor expansion
\begin{equation} \label{jn}
 \sum_{\s\in\perm_\ell}\pi_\ell(\s)\psi(\D^{\s_\ell\cdots\s_1}x) = t^\ell
X_w^\sharp \psi(x) + R_{\ell+1}(t^{\ell+1},\psi,x)\quad\text{for all $\psi\in
C^{\ell,1}_{\H,\loc}$.}
\end{equation}
 We will work  by induction. The statement for $\ell=2$ follows
immediately from the Taylor formula  \eqref{ttt}. Indeed
\begin{equation*}
\begin{aligned}
\psi(\Delta^{kj}x) &
=\psi(x)+ t(X_j^\sharp \psi+ X_k^\sharp \psi)(x) 
\\&\quad+
\frac{t^2}{2}\big((X_k^\sharp )^2\psi +
(X_j^\sharp )^2\psi + 2 X_j^\sharp  X_k^\sharp \psi \big)(x) + R_3(t^3,\psi,x),
\end{aligned}
\end{equation*}
 where  $j,k\in \{1,\dots,m\}$.
Thus $\psi(\D^k\D^j x ) - \psi (\D^j\D^k x) = t^2 X_{jk}^\sharp\psi(x)
+
R_3(t^3,\psi,x).$

Let us assume that \eqref{jn} holds for some $\ell\in\{2,\dots, s-2\}$.
Looking at the
Taylor expansion \eqref{ttt}, this means that
\begin{equation*}
 \label{babao}
\sum_{\s\in\perm_\ell}\sum_{\abs{\a}=0}^\ell
\pi_\ell(\s)\frac{t^{\abs{\a}}}{\a!}
(X_{\s_1}^\sharp)^{\a_1} \cdots (X_{\s_\ell}^\sharp)^{\a_\ell}\psi(x) = t^\ell
X_w^\sharp\psi(x)\quad\text{for all $t,x$ and $\psi\in C^{\ell,1}_{\H,\loc}$}.
\end{equation*}
In particular, if $k\le \ell-1 $, we have
\begin{equation}\label{wint}
\sum_{\abs{\a}=0}^{k} \frac{t^{\abs{\a}}}{\a!}
\sum_{\s\in\perm_\ell}\pi_\ell(\s)
(X_{\s_1}^\sharp)^{\a_1} \cdots (X_{\s_\ell}^\sharp)^{\a_\ell}\psi(x) =0
\quad\text{for all $t,x$.}
\end{equation}
In order to prove the induction step,  let $\psi\in C^{\ell+1,1}_{\H,\loc}$.
Then, omitting all the $\sharp$ symbols
\begin{equation*}
\begin{aligned}
&\sum_{\wt\s\in\perm_{\ell+1}}
  \pi_{\ell+1}(\wt\s)\psi(\D^{\wt\s_\ell\cdots\wt\s_1\wt\s_0 }x) \\& =
\sum_{\s\in \perm_\ell}\pi_\ell(\s)
\Big( \psi(\D^{\s_\ell\cdots\s_1 w_0 }x) -
\psi(\D^{w_0\s_\ell\cdots\s_1  }x)\Big)
\\&=\sum_{  \abs{\a}+\b=0}^{\ell+1}\frac{t^{\abs{\a}+\b}}{\a! \b!}\sum_{\s\in
\perm_\ell} \pi_\ell(\s)\big(X_{w_0}^\b X_{\s_1}^{\a_1} \cdots
X_{\s_\ell}^{\a_\ell}\psi(x) -
 X_{\s_1}^{\a_1} \cdots
X_{\s_\ell}^{\a_\ell} X_{w_0}^\b\psi(x)
  \big)
\\&\quad+R_{\ell+2}(t^{\ell+2},\psi,x)
\\&= \sum_{\abs{\a}=0}^{\ell+1} \frac{t^{\abs{\a}}}{\a!}
 \sum_{\s\in \perm_\ell}\pi_\ell(\s)
\big(  X_{\s_1}^{\a_1} \cdots
X_{\s_\ell}^{\a_\ell}\psi(x) -
 X_{\s_1}^{\a_1} \cdots
X_{\s_\ell}^{\a_\ell} \psi(x)
  \big)
\\&\quad +
 t \sum_{\abs{\a}=0}^\ell \frac{t^{\abs{\a}}}{\a!}
\sum_{\s\in \perm_\ell}\big(X_{w_0} X_{\s_1}^{\a_1} \cdots
X_{\s_\ell}^{\a_\ell}\psi(x) -
 X_{\s_1}^{\a_1} \cdots
X_{\s_\ell}^{\a_\ell} X_{w_0}\psi(x)
  \big)
\\&\quad +\sum_{\b=2}^{\ell+1}\frac{t^\b}{\b!}
\sum_{\abs{\a}=0}^{\ell+1-\b}\frac{t^{\abs{\a}}}{\a!}
\sum_{\s\in \perm_\ell}\pi_\ell(\s)
\big(X_{w_0}^\b X_{\s_1}^{\a_1} \cdots
X_{\s_\ell}^{\a_\ell}\psi(x) -
 X_{\s_1}^{\a_1} \cdots
X_{\s_\ell}^{\a_\ell} X_{w_0}^\b\psi(x)
  \big)
\\&\quad +R_{\ell+2}(t^{\ell+2},\psi,x).
\end{aligned}
\end{equation*}
Now note that the first line (case $\b=0$) vanishes trivially. The third
line, where $\beta\ge 2$,  vanishes
by virtue of \eqref{wint} (note that $X_{w_0}^\b\psi\in
C^{\ell+1-\b,1}_{\H,\loc}$).
It remains the term with  $\beta=1$ which   gives
\begin{align*}
  \sum_{\wt\s\in\perm_{\ell+1}}
  \pi_{\ell+1}(\wt\s)\psi(\D^{\wt\s_\ell\cdots\wt\s_1\wt\s_0 }x)  &=
t^{\ell+1}(X_{w_0}^\sharp X_w^\sharp \psi(x)- X_w^\sharp X_{w_0}^\sharp \psi(x))
+R_{\ell+2}(t^{\ell+2},\psi,x)
\\&=t^{\ell+1}X_{w_0 w}^\sharp\psi(x) +R_{\ell+2}(t^{\ell+2},\psi,x),
\end{align*}
by definition of $X_{w_0w}^\sharp$.
\end{proof}

\begin{proof} [Proof of Theorem \ref{dass}]
 We first show that (i) and (ii) are   equivalent for $\ell=2,\dots, s-1$. The
statement is obvious if $\ell=1$. Let now $\ell\in \{2,\dots, s-1 \}$
and take  $Z=
\psi\cdot\nabla\in C^{\ell, 1}_{\cal{H}}\cap
C^1_\Eucl$. Fix also  $w\in \W_\ell$.
Comparing the definitions
$\ad_Z X_w := ( Z^\sharp f_w - X_w \psi)\cdot\nabla$
 and $[Z, X_w]:=( Z^\sharp f_w -
X_w^\sharp \psi)\cdot\nabla$, we immediately recognize that  (i) and (ii) are equivalent.

Next we prove that $(i)$ holds for all $\ell\in\{2,\dots,s-1\}$.  In view of Lemma \ref{lemmauno},   it suffices
to prove that
for all $w=w_1\cdots w_\ell\in\W_\ell$, we have
\begin{equation}\label{labello}
\begin{aligned}
  \lim_{t\to 0}\sum_{\s\in\perm_\ell}\frac{1}{t^\ell}
\pi_\ell(\sigma)\psi(\D^{\s_\ell(w) \cdots \s_1(w)}x)
 =\sum_{\s\in \perm_\ell}
\pi_\ell(\s) X_{\s_1(w)}\cdots X_{\s_{\ell-1}(w)}f_{\s_\ell(w)} (x)
\cdot\nabla\psi(x),
\end{aligned}
\end{equation}
for any $\psi\in C^{\ell ,1}_{\cal{H},\loc}\cap C^1_\Eucl$ and for all
$\ell=2,3,\dots, s-1$.

We first prove the statement for $\ell=2$. Fix $X_{w_1}, X_{w_2}\in\{X_1, \dots,
X_m\}$  and    $\psi  \in C^{2,1}_{\cal{H},\loc}\cap C^1_\Eucl$. We need to
show
that
\begin{equation}
 \label{giogo}
\lim_{t\to 0}\frac{1}{t^2}\sum_{\s\in \perm_2}
\pi_2(\s)\psi(\D^{\s_2(w)}\D^{\s_1(w)}x)=X_{w_1w_2} \psi(x)
\quad \text{for all $x\in
\R^n$},
\end{equation}
where $X_{w_1w_2}:= (X_{w_1}
f_{w_2} - X_{w_2} f_{w_1})  \cdot \nabla$. Observe that since trivially
$X_k\psi= X_k^\sharp \psi$ for all $\psi\in
C^1_\Eucl$ and  $k=1, \dots, m$,
 we already have
\begin{equation}\label{radio}
 \ad_{X_k}X_i = X_{ki} \quad \text{for all $k,i\in\{1, \dots, m\}$.}
\end{equation}
For each fixed $x$, let  $g(t):= \sum_\s\pi_2(\s) \psi(\D^{\s_2}\D^{\s_1}x)$.
Here we take the abridged notation $\s_i=\s_i(w)$.
We will prove \eqref{giogo} by calculating the limit in the left-hand side  with
de~l'H\^opital's rule.
\begin{equation}\label{marelli}
\begin{aligned}
 g'(t) &= \sum_{\s\in\perm_2} \pi_2(\s)
\big\{X_{\s_1}(\psi\D^{\s_2})(\D^{\s_1}x) +
X_{\s_2}\psi(\D^{\s_2}\D^{\s_1}x)\big\}
\\&=
\sum_{\s\in\perm_2} \pi_2(\s) \big\{X_{\s_1}\psi (\D^{\s_2}\D^{\s_1}x)  +
X_{\s_2}\psi(\D^{\s_2}\D^{\s_1}x)
\\&\qquad \qquad + X_{\s_2\s_1}\psi (\D^{\s_2}\D^{\s_1}x)(-t) + O_3(t^2, \psi
,\Delta^{\sigma_2\sigma_1}x)\big\}.
\end{aligned}
\end{equation}
Here we already used Theorem
\ref{seminuovo} and we also invoked  \eqref{radio} to claim that
$\ad_{X_{\s_2}}X_{\s_1}=X_{\s_2\s_1}$.
To accomplish the proof for $\ell=2$, observe first  that
\begin{equation*}
 \lim_{t\to 0}\frac{1}{2t}\sum_{\s\in\perm_2}
\pi_2(\s) \Big( X_{\s_2\s_1}\psi
(\D^{\s_2}\D^{\s_1}x)(-t) + O_3(t^2, \psi, \Delta^{\s_2\s_1}x) \Big)
=X_{w_1w_2}\psi(x).
\end{equation*}
Here we used estimate \eqref{tredieci} and the definition of $\pi_2(\sigma)$.
Therefore the last line of \eqref{marelli} has the expected behaviour. It
remains to
show that the second one behaves as $O(t^2)$, as $t\to 0$.
To prove this claim, introduce
$\phi:=X_{w_1}\psi+
X_{w_2}\psi\in C^{1,1}_{\H,\loc}$.
 The
Taylor formula \eqref{ttt} gives
\begin{equation*}
\begin{aligned}
 \phi(\D^{w_2}\D^{w_1}x) & - \phi(\D^{w_1}\D^{w_2} x)
= \phi(x) + (X_{w_2}\phi(x)+X_{w_1}\phi(x))t + R_2(t^2, \phi,x)
\\&\qquad  -\big\{\phi(x) + (X_{w_1}\phi(x)+X_{w_2}\phi(x))t + R_2(t^2, \phi
,x)
\big\}=R_3(t^2, \psi,x),
\end{aligned}
\end{equation*}
as $t\to 0$.
Therefore,
\begin{equation*}
\begin{aligned}
 \lim_{t\to 0} 
\frac{1}{t}\sum_{\s\in\perm_2}
\pi_2(\s) \Big(X_{\s_1}\psi
(\D^{\s_2\s_1} x)  + X_{\s_2}\psi(\D^{\s_2\s_1} x)
 \Big)& =
\lim_{t\to 0} \frac{1}{t}\sum_{\s\in\perm_2} \pi_2(\s)
\phi(\D^{\s_2\s_1} x)
=
0,
\end{aligned}
\end{equation*}
as desired.

Next we show the induction step (which  is not needed if $s\le 3$). Assume that $s\ge 4$ and that 
for some  $\ell\in\{3, \dots, s-1\}
$ we have
\begin{equation}
\label{assumo} X_v\phi = X_v^\sharp\phi \quad\text{for all $\phi\in
C^1_\Eucl\cap
C^{\ell-1,1}_{\H,\loc}$ \quad $|v|\le \ell-1$}.
\end{equation}
We want to show \eqref{labello} for all $\psi\in C^{\ell,
1}_{\cal{H},\loc}\cap
C^1_\Eucl$ and $|w|=\ell$.

Fix $x$ and let \( g(t) :=\sum_{\s\in \perm_\ell}\pi_\ell(\s)
\psi(\D^{\s_\ell}\cdots\D^{\s_1}x),\) where $\s_j = \s_j(w)$.
It suffices to show that   $\lim_{t\to
0}g(t)/t^\ell =X_w\psi(x)$. This
will follow by de l'H\^opital's rule, as soon as we prove that
\begin{equation}\label{boxee}
\lim_{t\to 0}\frac{g'(t)}{\ell t^{\ell-1}} =  X_w  \psi(x).
\end{equation}
To show \eqref{boxee}, observe first that  by  the induction
assumption we have
\begin{equation}\label{indotto}
 \ad_{X_j}X_v =  X_{jv} \qquad\text{for all $j\in\{1, \dots, m\}$ \quad $|v|\le
\ell-1$.}
\end{equation}
Now we  calculate
$g'(t)$ keeping \eqref{itero} into account.
\begin{equation*}
\begin{aligned}
 g'(t) & =
\sum_{\s\in\perm_\ell} \pi_\ell(\s)\sum_{j=1}^\ell X_{\s_j} \big(\psi
\D^{\s_\ell \cdots \s_{j+1}} \big)
\big(\D^{\s_j \cdots \s_1} x\big)
\\&=
\sum_{\s\in\perm_\ell} \pi_\ell(\s)
\sum_{j=1}^\ell
\sum_{0\le k_{j+1} +\cdots + k_\ell\le \ell -1} X_{\s_\ell^{k_\ell}  \cdots
\s_{j+1}^{k_{j+1}}
\s_j}\psi(\D^{\s_\ell\cdots \s_1}x)\frac{(-t)^{k_{j+1}+\cdots
+k_\ell}}{k_{j+1}! \cdots k_\ell !}
\\&\qquad
+ O_{\ell+1}(t^{\ell},\psi, \Delta^{\s_\ell\cdots\s_1}x).
\end{aligned}
\end{equation*}
In view of \eqref{indotto}, we can expand as in \eqref{itero} and use identity 
\(
 \ad_{X_{\s_\ell}}^{k_\ell}\cdots
\ad_{X_{\s_{j+1}}}^{k_{j+1}}
X_{\s_j} \psi
=  X_{\s_\ell^{k_\ell}  \cdots
\s_{j+1}^{k_{j+1}}
\s_j}\psi
,\) which is legitimate because  $k_\ell+\cdots + k_{j+1}\le \ell -1$, see \eqref{indotto}.
We may rearrange as 
\begin{equation*}\begin{aligned}
&g'(t)=
\sum_{\s\in\perm_\ell}
\pi_\ell(\s)\sum_{1\le i_1\le \ell}X_{\s_{i_1}}\psi(\D^{\s_\ell\cdots\s_1}x)
\\& + \sum_{\mu=2}^\ell(-t)^{\mu-1}\sum_{p=2}^{\mu}
\;\sum_{\substack{1+b_2+\cdots + b_p=\mu  \\ b_2, \dots, b_p\geq 1}}
\frac{1}{b_2!\cdots b_p!}
\sum_{\s\in \perm_\ell} \pi_\ell(\s)
\sum_{1\le i_1<\cdots<i_p\le \ell}
X_{\s_{i_p}^{b_p}\cdots \s_{i_2}^{b_2}\s_{i_1}}\psi
(\Delta^{\s_\ell\cdots \s_1}x)
\\&
    + O_{\ell+1}(t^\ell, \psi,\Delta^{\s_\ell\cdots\s_1}x)
\\&
   =: H_1(t) + \sum_{\mu=2}^{\ell-1}(-t)^{\mu-1} H_\mu(t)
+(-t)^{\ell-1}\sum_{p=2}^{\ell-1} h_{\ell,p}(t)
+
(-t)^{\ell-1}h_{\ell,\ell}(t)
\\&
  +O_{\ell+1}(t^{\ell},\psi, \Delta^{\s_\ell\cdots
\s_1}x).
\end{aligned}
\end{equation*}
Everywhere $\s_j$ stands for $\s_j(w)$.

The proof  of \eqref{boxee} will be a consequence of the following three facts.

\step{Fact  1.} We have
\begin{equation*}
 \lim_{t\to 0}\frac{(-t)^{\ell-1} h_{\ell,\ell}(t)}{\ell t^{\ell-1}} =
\frac{(-1)^\ell}{\ell} 
h_{\ell,\ell}(0)=X_w\psi(x).
\end{equation*}

\step{Fact 2.} For any $p\in\{1, \dots, \ell-1\}$, we have
\begin{equation*}
\lim_{t\to 0}\frac{(-t)^{\ell-1} h_{\ell,p}(t)}{t^{\ell-1}}
=
(-1)^{\ell-1} h_{\ell,p} (0)=0.
\end{equation*}

\step{Fact 3.}
We have
\begin{equation}
 \label{terzo}
\lim_{t\to 0 } \sum_{\mu=1}^{\ell-1}\frac{(-t)^{\mu-1
}H_\mu(t)}{t^{\ell-1}} =0.
\end{equation}
Facts  1,2, and 3 give  easily the proof of \eqref{boxee} and of the theorem.

To check \textit{Fact 1}, just  observe that  property \eqref{tacin}
and the generalized Jacobi identity~\eqref{J2} give
\begin{equation*}
\begin{aligned}
\lim_{t\to 0}& \frac{(-t)^{\ell-1}}{\ell t^{\ell-1}}h_{\ell,\ell}(t)
= \frac{(-1)^{\ell-1}}{\ell} h_{\ell,\ell}(0)=
   \frac{(-1)^{\ell-1}}{\ell }
\sum_{\s \in\perm_\ell}\pi_\ell(\s)X_{\s_\ell\cdots\s_1}\psi (x)=
 X_w \psi(x),
\end{aligned}
\end{equation*}
as desired. Note that we used $\lim_{t\to 0}h_{\ell,\ell}(t)=h_{\ell,\ell}(0)$.

To verify  \emph{Fact   2}, note first that $\lim_{t\to 0}h_{\ell,p}(t)=h_{\ell,p}(0)$. Thus 
\begin{equation*}
\begin{aligned}
 h_{\ell,p}(0) & =
\sum_{\substack{ 1+b_2+\cdots + b_p= \ell
                                 \\  b_2, \dots, b_p\ge 1}}
\frac{1}{b_2!\cdots b_p!}
\Big\{ \sum_{\s\in\perm_\ell}\pi_\ell(\s)
\sum_{i\le i_1<\cdots < i_p\le \ell}X_{\s_{i_p}^{b_p}\cdots \s_{i_2}^{b_2}\s_{i_1}}
\psi(x)\Big\}
\\&
=  \sum_{\substack{ 1+b_2+\cdots + b_p= \ell
                                 \\  b_2, \dots, b_p\ge 1}}
\frac{1}{b_2!\cdots b_p!} 0,
\end{aligned}
\end{equation*}
because for any  $p\le \ell-1$ and $b_2,\dots, b_p\ge 1$,  the term
$\{\cdots\}$ vanishes by Proposition \ref{postace} below.

 Finally we discuss \emph{Fact 3}. Here it does not suffice to know that $H_\mu(t)\to
H_\mu(0)$, as $t\to 0$. We need instead a more refined expansion, whose explicit analysis 
is of considerable algebraic difficulty. 
Therefore, we use a slightly  more implicit
argument. First of all  we expand all  the terms  by means of  the  Taylor formula,
taking into account that
$\psi\in
C^{\ell,1}_{\cal{H},\loc}$. By inductive assumption  we may claim that   $
X_{\s_{i_p}^{b_p}\dots
\s_{i_2}^{b_2}\s_{i_1}} \psi =X_{\s_{i_p}^{b_p}\dots
\s_{i_2}^{b_2}\s_{i_1}}^\sharp\psi$. Thus we  express the latter as a sum of
horizontal derivatives of order  $ \mu$ with suitable coefficients. This gives
for $\mu\in\{2, \dots,
,\ell-1\}$,
\begin{equation*}
\begin{aligned}
 H_\mu(t) & = \sum_{p=1}^\mu \sum_{\substack{1+b_2+\cdots + b_p = \mu
\\ b_2,\dots,b_p\ge 1}}
\frac{1}{b_2!\cdots b_p!}\sum_\s\pi_\ell(\s)\sum_{1\le i_1<\cdots < i_p\le \ell}
X_{\s_{i_p}^{b_p}\dots \s_{i_2}^{b_2}\s_{i_1}}^\sharp \psi
(\D^{\s_\ell\cdots\s_1}x)
\\&= : \sum_{p,b,\s,i}\; \sum_{(k_1, \dots, k_\mu)\in \{w_1, \dots, w_\ell
\}^\mu }
c_{p,b,\s,i}^k X_{k_1 }^\sharp \cdots
X_{k_\mu}^\sharp \psi(\D^{\s_\ell\cdots\s_1}x)
\\&=\sum_{|\a|=0}^{\ell-\mu}
\;\sum_{p,b,\s,i}\;
\sum_{(k_1, \dots, k_\mu)\in \{w_1, \dots,
w_\ell\}^\mu }
\frac{c_{p,b,\s,i}^k}{\a
!}t^{|\a|}
(X_{\s_1}^\sharp)^{\a_1}\cdots (X_{\s_\ell}^\sharp)^{\a_\ell}
 X_{k_1}^\sharp \cdots X_{k_{\mu}}^\sharp \psi(x)
\\& \qquad
\qquad\qquad   + R_{\ell+1}(t^{\ell+1-\mu},
\psi),
\end{aligned}
\end{equation*}
where we also used the Taylor expansion. A similar expansion holds for  
$\mu=1$. 
Algebra of such coefficients is
quite complicated,  and it seems rather difficult to show \emph{Fact 3}
directly. We are instead able to prove what we need indirectly.
What we actually have is  a polynomial expansion of the form
\begin{equation}\label{abcd}
\begin{aligned}
& \sum_{\mu=1}^{\ell-1}(-t)^{\mu-1} H_\mu(t)
=\sum_{\lambda=1}^\ell t^{\lambda-1} P_\lambda
(X_{w_1}^\sharp, \dots, X_{w_\ell}^\sharp)
\psi(x) + R_{\ell+1}(t^\ell, \psi,x),
 \end{aligned}
\end{equation}
where $P_\lambda$ is an homogeneous polynomial of degree $\lambda$ involving
the coefficents $c_{p,b,\s,i}^k/\a!$ above.
Taking  \emph{Fact~1} and \emph{Fact~2}  for granted, this gives
\begin{equation}\begin{aligned}\label{ajaccio}
X_w^\sharp\psi(x)& = \lim_{t\to
0}\frac{g(t)}{t^\ell}\stackrel{\text{(H)}}{=}
\lim_{t\to 0}\frac{g'(t)}{\ell t^{\ell-1}}
\\&= X_w\psi(x) +\lim_{t\to 0}
\frac{1}{\ell t^{\ell-1}}\Big(\sum_{\mu=1}^{\ell-1}(-t)^{\mu-1} H_\mu(t) +
R_{\ell+1}(t^\ell, \psi)\Big)
\\&= X_w\psi(x)
  + \lim_{t\to 0}
\frac{1}{\ell t^{\ell-1}}\sum_{\lambda=1}^\ell t^{\lambda-1}
P_\lambda(X_{w_1}^\sharp, \dots, X_{w_\ell}^\sharp)\psi(x) 
\end{aligned}\end{equation}
  Equality $\stackrel{\mathrm{(H)}}{=}$ should be iterpreted in the usual
conditional sense provided by de~l'H\^opital's rule (the limit in the left-hand
side exists
and takes a value $L$ if the limit in the right-hand side
exists and takes the same  value $L$).
We do not know at this stage the value of the limit in the right-hand side. Our purpose is to show that 
it vanishes.

To prove such claim, note that equality \eqref{ajaccio} 
has an algebraic feature. Namely, all the coefficients
$c^k_{p,b,\s,i}$  appearing implicitely in the polynomials $P_\lambda$ do not
change if we take different vector fields $Z_j$ instead ov $X_{w_j}$ 
in some $\R^N$ with possibly
$N\neq n$, provided that we do not change the
number $\ell$ of vector fields.

If we choose analytic vector fields $Z_j$ in $\R^N$, we clearly have $Z_w\psi
 = Z_w^\sharp\psi$ for all $w$ and for any $\psi\in C^\omega$.
Moreover, the  conditional equality \eqref{ajaccio} becomes a true equality, because all
functions depend
analytically on $t$  and $x$.
Therefore we have  found a  family  of polynomial identities of the form
\begin{equation*}
0=P_\lambda(Z_1, \dots, Z_\ell)\psi(x) =: \sum_{(k_1,\dots,
k_\lambda)\in\{1,\dots,\ell \}^\lambda}
C(k_1, \dots, k_\lambda) Z_{k_1} \cdots Z_{k_\lambda}\psi(x)
\end{equation*}
which holds for any   family  $Z_1, \dots, Z_\ell$ of analytic  vector fields
in
$\R^N$, for each $N\in\N$, for all  analytic $\psi:\R^N\to \R$ and any
$x\in\R^N$.  Theorem \ref{amiz} shows that the polynomial should be  trivial, i.e.
$C(k_1, \dots, k_\lambda) = 0$ for all $(k_1, \dots, k_\lambda)$. This
concludes the proof of \emph{Fact 3} and of the theorem.
\end{proof}

Next we state and prove the relevant results needed to accomplish the proof of
\emph{Fact 2} and \emph{Fact 3}, that we took for granted in the argument above.

The  following family of nested commutators identities is relevant for the proof
of \emph{Fact~2}.
\begin{proposition}\label{postace}
 Let $X_1, \dots, X_m$ be   vector fields in the regularity class
$C^{s-1,1}_{\H,\loc}\cap C^1_\Eucl$. For any
$\ell\in\{2,\dots,s\}$
and  $1\le p\le \ell-1$,  we have the following statement:
\begin{equation*}\tag{$F_{\ell, p}$}\label{dido}
\begin{aligned}
 \sum_{\s\in\perm_\ell}  \pi_\ell (\s) &\sum_{1\le i_1<\cdots < i_{p}\le \ell}
X_{\s_{i_p}^{b_p}(v)\cdots \s_{i_2}^{b_2}(v)\s_{i_1}^{b_1}(v)w}=0
\quad \text{for all $ b_1,\dots, b_p\in \N\cup \{0\}$ }
\\&  |w|\geq 0 \quad \abs{v}=\ell  \quad  1\le b_1+\cdots+ b_p
\le s -\abs{w}.
\end{aligned}
\end{equation*}
\end{proposition}
We agree that if  $\abs{w}=0$, then  $X_{vw}= X_v$ for any  word $v$ with
$\abs{v}\ge 1$.  To prove \emph{Fact 2} we need the
case $|w|=0$ and $b_1=1$ of
the proposition,  but the case $\abs{w}=0$ is included for convenience in the
proof. 
Observe also  that
\begin{itemize*}\item if $|w|=0$ and $b_1\ge 2$, then the statement is
trivial;
\item  if $\ell=1$, then the statement is empty;
\item  Proposition \ref{postace} fails for $\ell=p$, as
\eqref{J2}
shows.
\end{itemize*}

Since the statement of Proposition \ref{postace} is quite intricated, we first
check its correctness   in the already significant case
$\ell=3$ and $p=2$, $\abs{w}=0$ and  $b\in \N$. The general case is based on
the same cancellation mechanism. In this model case, identity $(F_{3,2})$
becomes
\begin{equation*}
 \sum_{\s\in \perm_3}\pi_3(\s)\bigl\{X_{\s_3^b\s_2}+ X_{\s_3^b\s_1} +
X_{\s_2^b\s_1} \bigr\}=0,
\end{equation*}
which can be checked by  writing explicitly the twelve terms  (in the
notation
$[j^bk]:= X_{j^b k}$):
\begin{align*}
 [3^b2] + [3^b1] + [2^b1] & - \{[2^b3] + [2^b1]+ [3^b1]\}
\\ & -\{[1^b3]+ [1^b2] +[3^b2]\} + \{[1^b2]+ [1^b3]+ [2^b3]\}=0.
\end{align*}

\begin{proof}[Proof of Propoosition \ref{postace}]
 We first prove by induction that $(F_{\ell,1})$ holds for any $\ell\in\{
2,\dots,s\}$.
Introduce the abridged notation
   $[i_1^{b_1}i_2^{b_2}]$
instead of $X_{i_1^{b_1}i_2^{b_2}}$ and so on.
For convenience of notation, we prove $F_{\ell+1,1} $ for all $\ell\in
\{1,\dots, s-1\}$. Let $\wt v = v_0 v\in \W_{\ell+1}$, $w$ and $b_1\ge 1$ be
such
that $ b_1 + \abs{w}\le s$. Then
\begin{equation*}
\begin{aligned}
\sum_{\wt\s\in\perm_{\ell+1}}\pi_{\ell}(\wt \s) \sum_{0\le i_1\le
\ell}[\wt \s_{i_1}^{b_1 }(\wt
v)w] &
=\sum_{\wt\s\in\perm_{\ell+1}} \pi_{\ell}( \wt\s)
\big( [\wt \s_0^{b_1}(\wt v)w]+  [ \wt \s_1^{b_1}(\wt v)w] + \cdots+ [
\wt\s_\ell^{b_1}(\wt v)w]\big)
\\&= \sum_{\s\in\perm_\ell}\pi_\ell(\s)\big( [v_0^{b_1}w ]
+[\s_1^{b_1}(v)w]+\cdots+[\s_\ell^{b_1}(v)w]
\\&\qquad -\big([\s_1^{b_1}(v)w] + \cdots+[\s_\ell^{b_1}(v)w]+[
v_0^{b_1}w]\big)\big)
=0,\end{aligned}
\end{equation*}
as we claimed.

To fill up the triangle,   we prove that if $(F_{\ell,p-1})$ holds for some
$\ell\in\{2,\dots,s-1\}$ and $p\in\{2,\dots,\ell\}$,
then $(F_{\ell+1,p})$ holds.
This will imply that \eqref{dido}
holds for all the required couples $(p,\ell)$.
We argue as usual by the defining   property  \eqref{16}. Denote below $\wt v =
v_0 v\in \W_{\ell+1}$.
\begin{align*}
& \sum_{\wt \s\in\perm_{\ell+1}}\pi_{\ell+1}(\wt\s) \sum_{0\le i_1<\cdots<i_p\le
\ell}[\wt\s_{i_p}^{b_p}(\wt v)\cdots\wt\s_{i_1}^{b_1} (\wt v)w]
\\&=
\sum_{\s\in\perm_\ell} \pi_\ell(\s)\Big(\sum_{0\le i_1<\cdots<i_p\le
\ell}[\wt\s_{i_p}^{b_p}(\wt v)\cdots\wt\s_{i_1}^{b_1}(\wt v)w]\Big)\Big|_{\wt
\s(v_0 v) = v_0\s(v) }
\\&\qquad - \sum \pi_\ell(\s)\Big(\sum_{0\le i_1<\cdots<i_p\le \ell}
[\wt\s_{i_p}^{b_p}(\wt v)\cdots\wt\s_{i_1}^{b_1}(\wt v)w]
\Big)\Big|_{\wt \s(v_0 v) = \s(v) v_0}
\\&= \sum_{\s\in\perm_{\ell}} \pi_\ell(\s)
        \Big(\sum_{1\le i_1<\cdots<i_p \le \ell}
\hspace{-2ex}[\wt\s_{i_p}^{b_p}(\wt v)\cdots
\wt\s_{i_1}^{b_1}(\wt v)w]
+  \sum_{\substack{1\le i_2<\cdots< i_p\le \ell\\i_1=0}}
\hspace{-2 ex}[\wt\s_{i_p}^{b_p}(\wt
v)\cdots
\wt\s_{i_1}^{b_1}(\wt v)w]
\Big)\Big|_{\wt\s(\wt  v) = v_0\s(v)}
\\&  - 
 \sum_{\s\in\perm_\ell} \pi_\ell(\s)
        \Big(
\hspace{-2ex}
\sum_{0\le i_1<\cdots<i_p\le \ell-1}
\hspace{-2ex}
[\wt\s_{i_p}^{b_p}(\wt v)\cdots
\wt\s_{i_1}^{b_1}(\wt v)w]
+\hspace{-3ex}
\sum_{\substack{0\le i_1<\cdots< i_{p-1}\le \ell-1
\\  i_p=\ell}}
\hspace{-3ex}
[\wt\s_{i_p}^{b_p}(\wt v)\cdots \wt\s_{i_1}^{b_1}(\wt v)w]
\Big)\Big|_{\wt\s( \wt v) = \s(v) v_0}
\\&=
\sum_{\s\in\perm_\ell} \pi_\ell(\s)
\Big\{\sum_{1\le i_1<\cdots<i_p\le \ell}
\hspace{-2ex}
[\s_{i_p}^{b_p}(v)\cdots\s_{i_1}^{b_1}(v)w]
+\sum_{1\le i_2<\cdots<i_p\le\ell}
[\s_{i_p}^{b_p}(v)\cdots\s_{i_2}^{b_2}(v) v_0^{b_1} w]\Big\}
\\& - \Big\{\sum_{0\le i_1<\cdots<i_p\le\ell-1}
\hspace{-2ex}            
[ \s^{b_p}_{i_{p}+1}(v) \cdots\s^{b_1}_{i_{1}+1}(v)w ]+
\hspace{-2ex}
\sum_{0\le
i_1<\cdots< i_{p-1} \le \ell-1} 
\hspace{-2ex}
[v_0^{b_p} \s^{b_{p-1}}_{i_{p-1}+1}(v)
\cdots \s_{i_1+1}^{b_1}(v)w] \Big\}
\\&=0,
\end{align*}
because the first and the third term cancel, while both
the second and the fourth
 vanish by   inductive assumption.
\end{proof}

The following theorem has been used to check \emph{Fact 3} in the
proof
of Theorem~\ref{dass}. See~\eqref{terzo}.
\begin{theorem}\label{amiz}
 Let $m$ and $p$ be natural numbers. Let $C:\{1,\dots, m\}^p\to \R$ be given
coefficients. Consider the  polynomial
\begin{equation}
\label{cococo}
P(X_1, \dots, X_m):=  \sum_{(k_1, \dots, k_p)\in\{1,\dots, m \}^p}C(k_1, \dots,
k_p)X_{k_1}X_{k_2}\cdots X_{k_p}.
\end{equation}
Assume that for all $N\in\N$, 
for any $X_1, \dots, X_m$ analytic vector fields
 in $\R^N$ and  for each analytic  $\psi:\R^N\to \R$ we have
\begin{equation}
\label{abcabc}
 P(X_1, \dots, X_m)\psi(x)=0\quad\text{for all $x\in\R^N$.}
\end{equation}
Then $P$ is the trivial polynomial, i.e.  \(C(k_1, \dots, k_p)=0\)
 for all $ (k_1,
\dots, k_p)\in \{1,\dots, m \}^p.$
\end{theorem}
\begin{proof} The argument is inspired to some ideas contained in the proof the
Amitsur--Levitzki theorem \cite{Levitzki50,Amitsur50}.
  We start by separating homogeneous parts in each variable.
Let $N\in\N$ and take $X_1, \dots, X_m$  analytic vector
fields in $\R^N$ and $\psi$ analytic in $\R^N$.
 Consider the function
 \begin{equation*}
\begin{aligned}
  f(t_1, \dots, t_m): =& P(t_1 X_1, \dots, t_m X_m)\psi(x)
\\= :&
\sum_{q=1}^{\min\{p,m\}} \sum_{1\le i_1<\cdots<i_q\le m} \sum_{\substack{d_{ 1},
\dots, d_{ q}\ge 1 \\ d_{ 1}+\cdots + d_{ q}=p}}
t_{i_1}^{d_{ 1}}\cdots t_{i_q}^{d_{ q}} P^{i_1\cdots i_q}_{d_1\cdots
d_q}(X_{i_1}, \dots, X_{i_q})\psi(x),
  \end{aligned}
\end{equation*}
 where $x$ is  fixed. \footnote{An informal example to understand quickly this
splitting could be:
\begin{align*}
P(X_1, X_2, X_3)&= X_1^2 X_2^2 + X_3^4 + (X_1X_2X_3 X_1 + X_1^2 X_3
X_2)
\\&
=P_{2,2}^{1,2}(X_1, X_2) +  P_{4}^{3}( X_3)+
P^{1,2,3}_{2,1,1}(X_1, X_2, X_3).
\end{align*}
}
The function $f$
should vanish identically in $t_1, \dots, t_m$.
Therefore
it is clear that it must be for each fixed $q, i_1, \dots, i_q, d_1, \dots, d_q$
\begin{equation*}
 P^{i_1\cdots i_q}_{d_1\cdots d_q}(X_{i_1}, \dots, X_{i_q}) \psi(x)=0
\quad\text{for all $X_{i_1},\dots, X_{i_q},\psi \in C^\omega(\R^N)$\quad
$N\in \N$\quad $x\in \R^N$.}
\end{equation*}
In other words we can work with homogeneous polynomials in each variable.
Renaming variables, it
suffices to prove the theorem for a polynomial $P$  in $q$ variables, where
$1\le q\le p$ and such that
\begin{equation*}
 P(\la_1 X_1,
 \dots, \la_q X_q) =\la_1^{d_1}\cdots\la_q^{d_q} P(X_1,\dots ,X_q)
\qquad\text{for all $\lambda_1,\dots, \lambda_q\in\R$,}
\end{equation*}
where $d_1, \dots, d_q\ge 1$.

Next we show by a standard multilinearization argument that, possibly adding
new
variables, we can assume that $d_j=1$ for
all $j=1, \dots, q$. Indeed,  assume that $d_1\ge 2$.
Define
\begin{equation*}
 \wt P(U, T, X_2, \dots, X_q ):= P(U+ T, X_2, \dots, X_q )- P(U,  X_2, \dots,
X_q )
-P(T, X_2, \dots, X_q ).
\end{equation*}
It turns out that $\wt P$ is a homogeneous polynomial in $q+1$ variables, but
the degrees in the new variables $U$ and $T$ are both strictly less that the
original degree $d_1$.  Note that 
if $P(X_1, X_2, \dots, X_q)\psi\equiv 0$ for all $\psi, X_1, \dots, X_q\in C^\omega$, then
$\wt P(U, T, X_2, \dots, X_q)\psi\equiv 0$ for all $\psi,U, T, X_2,\dots, X_q\in C^\omega$.
On the other side,  if $\wt P$ is the  trivial polynomial (all its coefficients vanish),
 then also the polynomial $P$  must be trivial. Clearly, the polynomial $\wt P$ can be 
decomposed in a sum of 
homogeneous polynomials, where each of them is homogeneous in each variable separately, as above.

Iterating this argument we may assume that we  have a polynomial of the form
\begin{equation*}
Q(X_1, \dots, X_p )=\sum_{\s\in\perm_p}  B(\s) X_{\s_1}\cdots
X_{\s_p}
\end{equation*}
in $p$ variables, where $p$ is the original degree of the polynomial $P$
in \eqref{cococo}. Here $1 2\cdots p\mapsto\s_1\s_2\cdots\s_p$
are permutations and $\s_j = \s_j(12\cdots p)$.
We know that
\begin{equation}\label{zeropi}
 Q(X_1,  \dots, X_p )\psi(x)=0\quad\text{for all  $X_1,\dots,
X_p,\psi\in
C^\omega \quad N\in \N\quad x\in \R^N$}
\end{equation}
and we want to show that $B(\s)=0$ for all $\sigma$. 
Since we are free to increase the dimension $N$ of the underlying
space,  take $N\ge p+1$, let $X_j = x_j\p_{j+1}$ for any $j=1, \dots,
p$. Therefore, it
turns out that, if we let $\psi(x)= x_{p+1}$, we have
\begin{equation*}
 X_{\s_1}\cdots X_{\s_p} \psi = \begin{cases}
                            1 & \text{if $\s_1\cdots\s_p = 1\cdots p$;}
\\ 0 & \text{if $\s_1\cdots\s_p \neq 1\cdots p$.}
                           \end{cases}
\end{equation*}
Therefore, if we  make use of \eqref{zeropi}, we discover that   it must be
$B(1,2,\dots, p)=0$. Letting then $X_{\s_j} = x_{ j}\p_{
j+1}$ we see that $B(\s)=0$ for all $\s\in\perm_p$. Therefore $Q$ is the trivial
polynomial and   the proof is concluded.
\end{proof}

\subsection{An old nested commutators identity due to Baker}\label{chetbaker}
Here we show as an application that some very old nested commutator identities
going back to Baker (see the discussion in \cite{Oteo}) can be found as
a particular case
of our Proposition \ref{genero}.
All vector fields in this subsection are smooth.

Let $v = v_1
\cdots v_\ell  $ be a word of length $\ell$ in the
alphabet $1,\cdots ,m$. Let us  adopt the notation
\begin{equation*}
 X_{a }^k X_b:=
\begin{cases}
X_a X_b&\text{if $k=1$}
\\ X_b X_a &\text{if $k=-1$}                \end{cases}\text{\quad for all
$a,b\in \{1,\dots, m\}$,}
\end{equation*}
\begin{equation*}
 X_{a }^k X_b^h X_c:=
\begin{cases}
X_a X_b^h X_c&\text{if $k=1$ and $h\in\{-1,1\}$}
\\ X_b^h X_c X_a &\text{if $k=-1$ and $h\in\{- 1,1\}$}
\end{cases}\text{\quad for all
$a,b,c\in \{1,\dots, m\}$}
\end{equation*}
and analogous notation for higher order derivatives. Then  it is rather easy
to check that we may write for all $v= v_1\cdots v_\ell$
\begin{equation}\label{alticcio}
 X_{v_1\cdots v_\ell} = \sum_{k_1, \dots ,
k_{\ell-1}\in\{-1,1\}}(-1)^{k_1+\cdots +
k_{\ell-1}}X_{v_1}^{k_1}\cdots X_{v_{\ell-1}}^{k_{\ell-1}}X_{v_\ell}.
\end{equation}
This is an alternative way to write commutators, less focused on the inductive
point of view than the form \eqref{commototo}. Let now  $n\in\N$, $v\in
\W_{n+1}$ and
$w\in \W_1$. Thus,
\begin{align*}
 X_{wv}+\sum_{\sigma\in
\perm_{n+1}}\pi_{n+1}(\sigma)X_{\s_1(v)\cdots\s_{n+1}(v)w} =
 -X_{[v]w}+\sum_{\sigma\in
\perm_{n+1}}\pi_{n+1}(\sigma)X_{\s_1(v)\cdots\s_{n+1}(v)w} =0,
\end{align*}
by \eqref{otto}. Comparing   \eqref{alticcio} and \eqref{commototo},  this is equivalent to
\begin{equation*}
 X_{wv_1\cdots v_n v_{n+1}} + \sum_{k_1, \dots, k_n\in\{-1,1\}}(-1)^{k_1+\cdots
+
k_n}X_{v_1^{k_1}\cdots v_n^{k_n} v_{n+1},w}=0,
\end{equation*}
or in the typographically better, self-explanatory notation
\begin{equation}\label{giochetto}
[wv_1\cdots v_n v_{n+1}] + \sum_{k_1, \dots, k_n\in \{-1,1\}}(-1)^{k_1+\cdots +
k_n}[v_1^{k_1}\cdots v_n^{k_n} v_{n+1},w]=0.
\end{equation}
Note that  we introduced a comma before $w$ to avoid confusion. Namely,
when some of the $v_j$ has power $-1$, then it goes on the right side of the
previous $v_{j+1}, \dots, v_{n}, v_{n+1}$ but \emph{not} of $w$. For instance,
we have
$[v_1^{-1}v_2^{-1}v_3 v_4,w]:=[v_3v_4v_2 v_1 w] $ and so on (a precise
definition can be given by induction).

Now we show that the following Baker's identity
of order six
\begin{equation}\label{panettiere}
[ab^4 a]-2[bab^3a]+ [b^2 ab^2 a]=0 \quad\text{for all $a,b\in \{1, \dots, m\}$,}
\end{equation}
 see \cite[eq.~(4.4)]{Oteo}, can be easily obtained specializing
\eqref{giochetto}.
Let $n=4$ and choose $v_1\cdots v_n v_{n+1} = v_1\cdots v_4 v_5 = b\cdots b
a= b^4 a$ and $w=a$. Thus \eqref{giochetto} becomes
\begin{equation*}
 [a b^4 a]+ \sum_{k_1, \dots, k_4\in \{-1 ,1\}}(-1)^{k_1+\cdots + k_4}
[b^{k_1}\cdots b^{k_4}a, a]=0.
\end{equation*}
Note that at least one among the numbers $k_j$ must be $-1$, otherwise we get
$[b^4 aa]=0$. Therefore we get
\begin{equation*}
 [a b^4 a]+ \Bigl\{- \binom{4}{1} [b^3 ab a]+\binom{4}{2}[b^2 ab^2 a]-
\binom{4}{3}[bab^3 a]+\binom{4}{4}[ab^4a]\Bigr\}=0
\end{equation*}
which gives
$
 [ab^4 a]-2[b^3 aba]+3 [b^2 ab^2 a]-2[bab^3 a]=0.
$
But the fourth order identity \eqref{oteo} gives  $[b^3aba  ]= [b^2 ab^2 a]$.
Thus \eqref{panettiere} follows.

\section{Applications to ball-box theorems}\label{cinquemila} 
In this section we describe some applications of our results to ball-box theorems. 
We shall use some results from \cite{MontanariMorbidelli11d} 
  on the expansion of almost exponential maps.

      Assume that a family $\H$  of vector
 fields belongs to the regularity class $C^{s-1,1}_{\H,\loc}\cap C^1_\Eucl$
 and assume that the vector fields satisfy the H\"ormander condition of step $s$, namely 
$
      \dim \Span\{X_w(x)\colon 1\le\abs{w}\le s\}=n
$, at any $x\in\R^n$.  
Following a standard notation, denote by
$\cal{P}   := \{ Y_1, \dots, Y_q\} = \{ X_w: 1\le \abs{w} \le s\}$
the family of commutators of length at most $s$.  Let
$  \ell_j\le s$  be the length of $Y_j$ and
write $Y_j=:g_j\cdot\nabla$.
For each $I=(i_1, \dots, i_n) \in\{1,\dots,q\}^n$, let $\ell(I)= \ell_{i_1}+\cdots+\ell_{i_n}$
$
\lambda_I(x):= \det[Y_{i_1}(x),\dots, Y_{i_n}(x)]$ and
$ \ell(I):= \ell_{i_1}+\cdots+ \ell_{i_n}.$ 
Define also the vector valued function
\( 
      \Lambda(x,r):= (\lambda_I(x)r^{\ell(I)})_{I\in\{1,\dots,q\}^n}.
\) 
Finally, for all  $A\subset\R^n$, put
\begin{equation}
\label{nu}
 \nu(A) : = \inf_{x\in  A }  \abs{
\Lambda(x,1)}. \end{equation}

Assume that each  commutator $Y_j$  is continuous in the Euclidean topology. Then, on the open set
$\Omega_0\subset\R^n$ fixed before \eqref{lipo},  we have
\(\nu(\Omega_0)>0.\) 
Moreover, take  $j\in\{1,\dots, m\}$ and  any word $w$ with $\abs{w}=s$. For any $x\in\Omega_0$
 where the 
derivative $X_j^\sharp f_w(x)$ exists,  we have the  obvious bound $\abs{X_j^\sharp f_w(x)}\le L_0$, 
the constant in \eqref{lipo}. Furthermore we also have     $\abs{X_w f_j(x)} \le L_0$
for all $x$.
Therefore we can write 
\begin{align}\label{arte}
 \ad_{X_j} X_w (x) & = \sum_{1\le \abs u\le s}b^u X_{u}(x)\quad\text{where}
\\ \label{artefatto}  \abs{b^u } & \le C_0 \qquad\text{for all $u$ with $1\le
\abs{u}\le s$. }
\end{align}
Here the  constant $C_0$ can be estimated  in terms of the constant $L_0$ in \eqref{lipo}
 and of the infimum  $\nu(\Omega _0)$; see \cite[Lemma~4.2]{MM}. 

Therefore,
the vector fields are in the class $\A_s$ introduced in 
 in \cite{MontanariMorbidelli11d} 
(actually in a subclass, because here we assume the H\"ormander condition, while in  
\cite{MontanariMorbidelli11d}  we did not).
Moreover we have the following measurability property:

\begin{proposition}[measurability] \label{arcano} Let $\H$ be a family  of vector fields
in the regularity class $C^1_\Eucl\cap C^{s-1, 1}_{\H, \loc}$. Assume the H\"ormander condition at step $s$
and assume that   $f_w\in C^0_\Eucl$, if  $1\le \abs{w}\le s$.
 Let $w$ be a word with $\abs{w} =s$ and let
$Z=f\cdot\nabla \in \pm\H$.
 Then for any $x\in \Omega$ we can write
 \begin{equation}
 \ad_{Z} X_w(e^{t Z}x) = \sum_{1\le\abs{v}\le s} b^v(t)
X_v(e^{t  Z}x) \quad \text{for a.e. $t\in (-t_0, t_0)$,}
\end{equation}
where the functions $t\mapsto b^v(t)$ are \emph{measurable} and $\abs{b^v(t)}\le
C_0$, the constant in \eqref{artefatto}.
\end{proposition}
\begin{proof}  Denote
  $\gamma(t):= e^{tZ}x$. Since $t\mapsto f_w(\gamma(t))$ is
Lipschitz and $x\mapsto X_w
f(x)$ is continuous, the function $t\mapsto \ad_Z X_w(\gamma(t)):=
\frac{d}{dt}f_w(\gamma(t))  - X_w f(\gamma(t))$ is measurable and  bounded, as observed above.  Let
for any $x$ the matrix $Y_x= [Y_{1,x},\dots, Y_{q,x}]\in \R^{n\times q}$. Then
let
$Y_x^\dag$ be the Moore--Penrose pseudoinverse of $Y_x$.  Therefore, choose
$
b(t) = Y^\dag_{\gamma(t)}(\ad_Z X_w)_{\gamma(t)}
$ at any differentiability point  $t$. Note that $b(t)$ is the least-norm
solution of the system $\sum_{j=1}^q Y_{j, \gamma(t)}\xi^j =
(\ad_ZX_w)_{\gamma(t)}$, where $\xi\in\R^q$.
The Tychonoff approximation  $Y^\dag=\lim_{\d\downarrow 0} (Y^T Y
+ \d I_q)^{-1}Y^T $ (see   the appendix)
shows measurability.
\end{proof}

\begin{remark}\label{misurella}
\begin{itemize*}
\item 
One can prove Proposition \ref{arcano} in a less elegant but  more analytic way,  without using the Moore--Penrose inverse, looking instead for ``almost least-squares'' solutions.

\item

The argument  above
can be used to see that  in the definition of subunit distance we
may work with paths $\gamma$ such that  for a.e.~$t$ we have $\dot \gamma(t) =
\sum_{j}b^j(t) X_j(\gamma(t))$, where the function $t\mapsto b(t)$ is
\emph{measurable}.
Indeed,
let $\gamma$ be a subunit path as in  the definition of $\dcc$ in
\eqref{dicacco}.
Given    a differentiability point $t$ of  $\g$, let
$
   b (t):=\lim_{\d\downarrow 0} \big(X_{\g(t)}^T X_{\g(t)}+\d
I_p\big)^{-1}X_{\g(t)}^T\dot\gamma(t),
$
where $X_{x}:=[X_{1,x}, \dots, X_{m,x}]$ for all $x$.
The function $  b $ is measurable and at any  differentiability point $t
 $ of $\gamma$, the vector $  b(t)$ is the least-norm solution of the system
$X_{\g(t)}\xi = \dot\gamma(t)$, with $\xi\in \R^m$.
  See \cite{JerisonSanchez} for a related discussion.
      
\end{itemize*}

\end{remark}

The distance associated  with $\P$ where each $Y_j$ has degree $\ell_j$ will be
denoted by $\r$:
\begin{equation}\label{coscos}
\begin{aligned}
 &\text{$\r(x,y) := \inf \big\{  r\ge 0 :   $ there is $\gamma\in
\Lip_\Eucl((0,1), \R^n)$ with $\gamma(0)=x$   }
\\& \quad\text{$\gamma(1) = y$ and $\dot\gamma(t) ={\textstyle{
\sum_{j=1}^q}} b_j r^{\ell_j}Y_j(\g(t))$ with $\abs{b}\le 1$ for a.e. $t\in
[0,1]$}\big\}.
\end{aligned}
\end{equation}

Next we recall the definition of approximate exponential.
Let $w_1, \dots, w_\ell\in \{1,\dots,m\}$.
Given $\t>0$, we define, as in \cite{NagelSteinWainger,Morbidelli98} and
\cite{MM},
\begin{equation}\label{navetta}
 \begin{aligned}
 C_\t( X_{w_1})& := \exp(\t X_{w_1}),
 \\ C_\t( X_{w_1}, X_{w_2})& :=\exp(-\t X_{w_2})\exp(-\t X_{w_1})\exp(\t
X_{w_2})\exp(\t X_{w_1}),
 \\&\vdots
  \\C_\t( X_{w_1}, \dots, X_{w_\ell})&
:=C_\t( X_{w_2}, \dots, X_{w_\ell})^{-1}\exp(-\t X_{w_1}) C_\t( X_{w_2}, \dots,
X_{w_\ell})\exp(\t X_{w_1}). \end{aligned}
 \end{equation}
Then let
\begin{equation}
\eap^{tX_{w_1 w_2\dots w_\ell}} :=  \expap(t X_{{w_1 w_2\dots w_\ell}}):=
\left\{\begin{aligned}
& C_{t^{1/\ell}}(X_{w_1}, \dots, X_{w_\ell} ), \quad &\text{ if $t\geq 0$,}
\\
&C_{|t|^{1/\ell}}(X_{w_1}, \dots, X_{w_\ell} )^{-1}, \quad &\text{ if $t<0$.}
                 \end{aligned}\right.
\label{approdo}
\end{equation}
Let $\Omega_0$ be the open bounded set fixed before \eqref{lipo}. 
By standard ODE theory,    there is $t_0$  depending on $\ell,  \Omega$,
$\Omega_0$,  $ \sup\abs{ f_j } $ and  $\sup\abs{\nabla f_j}  $     such that
$\exp_*(t X_{{w_1 w_2\dots w_\ell}})x\in\Omega_0$ for any $x\in
 \Omega$ and $|t|\le t_0$.
  Given $r>0$, define $\wt Y_j = r^{\ell_j}Y_j$ for $j=1,\dots,q$. Moreover,  
if $I= (i_1,\dots,i_n)\in\{1,\dots,q\}^n $,   $x\in \Omega$, $r\in (0,1]$ and
$h\in\R^n$ is sufficiently close to the origin,  define
     \begin{equation}
\begin{aligned} \label{hhh} E_{I,x,r}(h)& :=\expap(h_1 \wt Y_{i_1})\cdots
\expap(h_n
\wt Y_{i_n})(x)
\\
\bigl\|h\bigr\|_I & : =\max_{j=1,\dots,n}|h_j|^{1/\ell_{i_j} }\qquad   Q_I(r):
=\{h\in\R^n:\norm{h}_I < r\}.
\end{aligned}\end{equation}

Recall that, given  $\eta\in(0,1)$,  $x\in K$,  $r<r_0$ and  $I\in\{1,\dots,q\}^n$,
the triple  $(I,x,r)$ is said to be  
$\eta$-maximal if
\(
\abs{\lambda_I (x)} r^{\ell (I)} >\eta \max_{J\in \I(p_x, q)} \abs{\lambda_J(x)}r^{\ell(J)}.
\)

\begin{theorem}\label{teoremahormander} 
      Let $\H$ be a family of  vector fields 
of class $C^{s-1,1}_{\H,\loc}\cap C^1_\Eucl$  satisfying the H\"ormander condition of step $s$. Assume that all nested commutators up to length $s$ are continuous in the Euclidean sense.
Then there is $C>1$ such that the following properties hold.
Let $I\in\{1,\dots,q\}^n$, $x\in\Omega$ and $r<C^{-1}$. Let also  $E:=E_{I,x,r}$ be the map  in~\eqref{hhh}. Then
\begin{enumerate*}
\item [(a)]   $E\in C^1_\Eucl(Q_I(C^{-1}))$.
\item[(b)]\label{bibibo}  We have the expansion  
\begin{equation}\label{polacco}
\begin{aligned}\frac{\p}{\p h_k} E(h) =   \wt Y_{i_k}(E(h))  +\sum_{\ell_j=\ell_{i_k}+1}^{s}
a^j_k(h)
\wt Y_{j  }(E(h))
+ \sum_{i=1}^q  \omega_k^i(x,h)  \wt Y_{i}(E(h)).
\end{aligned}\end{equation}
where $\wt Y_k:= r^{\ell_k}Y_k$ and the functions $a_k^j$ and $\o_k^j$ satisfy
\begin{align}\label{sogliola}
       \abs{a_k^j (h)}& \le C\bigl\|h \bigr\|_I^{\ell_j-\ell_{i_k}}\quad\text{for all
$h\in Q_I(C^{-1})$}
\\ \label{merluzzo}  \abs{\omega_k^j(x,h)} &\le C \bigl\|h \bigr\|_I^{s+1-  \ell_{i_k}}
\quad\text{for all $h\in Q_I(C^{-1})\quad x\in\Omega$}.
\end{align} 
\item [(c)]
If  moreover $(I, x, r)$ is   $\frac12 $-maximal with $I\in\{1,\dots,q\}^n$, $x\in \Omega$ and  $r< r_0$,
then,       for  all  $\e\leq C^{-1}$ we have
\begin{equation}\label{moma}
 E_{I,x,r}(Q_I(\e))\supset  B_\rho (x, C^{-1} \e^s r).
\end{equation}
\end{enumerate*}
\end{theorem}
Note that constants in  Theorem  \ref{teoremahormander}  depend quantitatively on $C_0$ and $L_0$. 
 Inclusion~\eqref{moma} ensures the 
 Fefferman--Phong type estimate $d(x,y) \le  C\abs{x-y}^{1/s}$; see \cite{FeffermanPhong81}.

Moreover, we have
\begin{theorem}\label{innere}
Assume that  the hypotheses of Theorem \ref{teoremahormander} hold.
Then there is is a
constant $C>0$  such that the following holds.
  Let    $x\in\Omega\Subset\Omega_0$. Then,  for any $\frac 12$-maximal triple  $(I,x,r)$ with 
$I\in\{1,\dots,q\}^n$,
$x\in\Omega$  and $r<C^{-1}$, the map
 $E_{I,x,r}$ is one-to-one on the set $Q_I(C^{-1})$.
\end{theorem}
The constant $C$ in Theorem \ref{innere} does not depend quantitatively on $C_0$ and $L_0$, because  \ref{finallyy} below involves a qualitative covering argument. 
A more precise control on such constant can be obtained assuming more regularity (for instance if the vector fields belong to the class $\B_s$ of \cite{MontanariMorbidelli11c}).


\begin{proof}[Proof of Theorems \ref{teoremahormander} and \ref{innere}]
All arguments of the proofs are contained in the papers \cite{NagelSteinWainger,Morbidelli98,MM,MontanariMorbidelli11d} and \cite{MontanariMorbidelli11c}.  Let us recapitulate the
 skeleton of the proof with precise references to the mentioned papers.

\begin{enumerate}[noitemsep,label=(\roman*)]
      
 \item \label{uutt}
  Specializing
 \cite[Remark~3.3]{MontanariMorbidelli11d} 
to our setting, we may claim 
that if $(I,x,r)$ is $\eta$-maximal, then $(I,y,r)$ is $C^{-1}\eta$-maximal for all $y\in B_d(x,C^{-1}\eta r)$. 
\item \label{uutt2}  The proof of  
  Theorem~\ref{teoremahormander}, items 
  (a)
and (b) are contained in
 \cite[Theorem~3.11]{MontanariMorbidelli11d}.  
Note 
that the mentioned 
result holds even in a more general setting where the H\"ormander's rank condition is not assumed.

\item \label{jojojo} In view of \ref{uutt},  \ref{uutt2} and expansion \eqref{polacco} 
we can follow the proof of \cite[Lemma~5.14]{MM} (just letting $\sigma=0$). Thus, we may claim that if $\xi\in\Omega$ and  $\abs{\lambda_I(\xi)}\neq 0$, then  $E_{I,\xi,r}$ is one-to-one on $Q_I(C^{-1}r^{\ell(I)} \abs{\lambda_I(\xi)})$.

\item  \label{orbuz}  For all $\eta\in (0,1)$ there is $C_\eta>0$ such that given an $\eta$-maximal triple
 $(I,x,r)$, then  the map  $E_{I,x,r}$ satisfies for all $j\in\{1,\dots,n\}$ the expansion
\begin{equation}\label{bonfo}
  \frac{\p }{\p h_j} E(h)= \wt Y_{i_j} (E(h)) + \sum_{1\le k\le n} \chi_j^k(h)
\wt Y_{i_k}(E(h))
\quad\text{for all $h\in Q_I (C_\eta^{-1})$,}
\end{equation}
where $\chi\in C^0_\Eucl(Q_I(C_\eta^{-1}) , \R^{n\times n})$ satisfies
\begin{equation}\label{dichi}
 |\chi (h)|\le C_\eta\norm{h}_I\qquad \text{if  $\norm{h}_I\le C_\eta^{-1}$.}
\end{equation}
Therefore, for a suitable $\wt C_\eta$ possibly larger that $C_\eta$,  the map $E_{I,x,r}\big|_{Q_I(\wt C_\eta^{-1})}$ is a local $C^1$ diffeomorphism and in particular it is  open. This ensures that the 
topologies of the distances  $\r, \dcc$ and $d$ are all locally equivalent to the Euclidean one.  
Expansion~\eqref{bonfo}  with estimate~\eqref{dichi} has been proved in \cite[Theorem~3.1]{MontanariMorbidelli11c}.
As observed after the statement in \cite{MontanariMorbidelli11c}, such result holds in the broader class $\A_s$.
Note that in \cite{{MontanariMorbidelli11c}} we discuss the case $\eta=\frac 12$. The case with   $\eta<\frac 12$ can be treated with minor modifications.

\item \label{III}  To prove Theorem  \ref{teoremahormander}-(c), it suffices to follow the 
proof of \cite[Lemma~3.7]{MontanariMorbidelli11c}.  
This is explained in \cite[Remark~3.8]{MontanariMorbidelli11c}.

\item \label{finallyy}  Finally, keeping all previous items into account,  to prove the injectivity result Theorem \ref{innere}, it suffices to follow   \cite[pp.~132--133]{NagelSteinWainger} or  \cite[Lemma~3.6]{Morbidelli98}.
In the proof of the latter
lemma,  note that in third line of  \cite[Eq.(30)]{Morbidelli98},   which reads
\[\abs{\lambda_{I_{0,k}}(x)}\d_{0,k}^{d(I_{0,k})}>\frac 12\max_J\abs{\lambda_{J}(x)}\delta_{0,k}^{d(J)} 
\quad\text{for all $x\in U_k$,}\] 
by \ref{uutt} we may choose $U_k = B_d(x_k, C^{-1}\delta_{0,k})$, which is open by \ref{orbuz}; moreover,    by \ref{jojojo} we may assume that $E_{I,x,\delta_{0,k}}\big|_{Q_I(\delta_{0,k})}$ is one-to-one for each $x\in U_k$. 
The remaining part of the proof in \cite{Morbidelli98} can be applied verbatim to our setting.

\end{enumerate}
\end{proof}

\begin{remark}
Theorem \ref{innere} implies the \emph{doubling property} for vector fields satisfying the hypotheses of Theorem \ref{teoremahormander}. Let $ \Omega\subset\R^n$ be a bounded 
 open set. Then there are $C$ and $r_0>0$ so that  
\begin{equation*}
      \abs{B_\textup{cc}(x,2r)}\le C\abs{B_\textup{cc}(x,r)}
\quad\text{for all $x\in \Omega$ $r< r_0$.}
\end{equation*}
Moreover, following \cite{LM}, one gets  for all $f\in C^1(B_\textup{cc}(x,Cr))$, the \emph{Poincar\'e inequality}
\begin{equation*}
      \int_{B_\textup{cc}(x,r)}\abs{f(y) 
- f_B}dy\le C r\int_{B_\textup{cc}(x, Cr)}
\sum_j \abs{X_jf(y)}dy\quad\text{for all $x\in \Omega$ $r< r_0$.}
\end{equation*}
Finally, as in \cite[Proposition~6.2]{MM},  given
$\Omega'\subset\subset\Omega$,  and $\e\in\left]0,1/s\right[$, there is $
r_0$ and $C>0$ 
such that, for any $f\in C^1(\Omega)$,  \begin{equation}\label{orol}
  \int\limits_{\substack{{\Omega'\times\Omega'}\\
   {d(x,y)\le  r_0}}}\frac{|f(x)-f(y)|^2}{|x-y|^{n+2\e }}dxdy\le C\int_\Omega \sum_j|X_jf(y)|^2 dy.
 \end{equation}
\end{remark}

\appendix
\section{Appendix}
\label{bio}
\paragraph{Tychonoff regularization for the Moore--Penrose
pseudoinverse.}
Here we discuss an approximation  formula  for the Moore--Penrose inverse  of a
matrix which has been used in Proposition~\ref{arcano}. This result  is proposed as an exercise in some matrix-analysis
textbooks  (see \cite[Problem 5.5.2]{Golub}).  We include here  a short
discussion for completeness.

Let $a_1, \dots, a_q\in \R^n$ and let $A= [a_1, \dots, a_q]\in \R^{n\times q}$.
Take $b\in \Span\{a_1, \dots, a_q\}$ and look at the system $Ax=b$ where $x\in
\R^q$. We do not assume that the vectors  $a_j$ are independent. Let
$\xls $ be
the solution of minimal norm. We claim that
\begin{equation}\label{elleesse}
\xls = \lim_{\lambda\to 0}(A^T A + \lambda^2 I_q)^{-1}A^T b.
\end{equation}
In other words, the family if matrices $(A^T A + \lambda^2 I_q)^{-1}A^T $
gives an approximation of the Moore--Penrose inverse $A^\dag$, as $\lambda\to
0$.
 Note that,
if $a_1, \dots, a_q$ are  independent, then it is well known that $A^\dag =
(A^TA)^{-1} A^T$. If they are dependent, then $A^TA$
is singular, but  still we have $\lim_{\lambda\to 0}
(A^T A + \lambda^2 I_q)^{-1}A^T= A^\dag$.

To show \eqref{elleesse},  write $A= U\Sigma V^T$ as a
 singular value decomposition, i.e.
$U\in O(n)$,   $V\in O(q)$, while $\Sigma =\diag(\s_1, \dots, \s_r, 0, \dots,
0)\in \R^{n\times q}
$, where $\s_1\ge \cdots \ge \s_r >0$ are the singular values of $A$ and $r\le
\min\{q,n\}$ is
its rank.  Note that $U^T[a_1,\dots, a_q]= \Sigma V^T$. Therefore,
 $U^T a_j\in \R^r\times \{0_{n-r}\}$ and $U^T b\in \R^r\times
\{0_{n-r}\} $, too.

By definition, the vector  $x\in \R^q$ is a (not unique) least-square  solution
of the system $Ax=b$
if and only if it solves $ A^TA x = A^T b $, which is equivalent to
$\Sigma^T\Sigma V^T x =
\Sigma^T U^T b$, or, letting $V^T x= :\xi$ and $  U^Tb =:
\beta\in\R^n$, to the system
\begin{equation}\label{dibeta}
\Sigma^T\Sigma \xi
= \Sigma^T \beta.
\end{equation}
Since $\Sigma^T\Sigma = \diag(\s_1^2, \dots,\s_r^2 , 0, \dots, 0)\in
\R^{q\times q}$ and since the system \eqref{dibeta} has solutions by
assumptions on the data $b$, it must be $\Sigma^T\beta =(\s_1 \b_1,\dots,
\s_r\b_r,0,\dots)^T\in \R^{q}$ and  the solutions of \eqref{dibeta}
are $\xi= (\b_1/\s_1,\dots, \b_r/\s_r,\xi_{r+1}, \dots, \xi_q)^T$, with
$\xi_{r+1}, \dots, \xi_q$ free parameters.
Clearly, the minimal-norm one is
$\xi_{\textup{LS}}= (\b_1/\s_1, \cdots, \b_r/ \s_r, 0,
\dots )^T\in \R^q$.

Define now  the vector
$x_\lambda:=(A^TA+\lambda^2 I_q)^{-1}A^T b = V(\Sigma^T \Sigma+ \lambda^2
I_q)^{-1}\Sigma^T U^T b$. Since  $\Sigma^T U^T b=\Sigma^T\beta= (\s_1\b_1,
\dots, \s_r\b_r, 0, \dots)^T$, we have
\[V^T x_\lambda=:\xi_\lambda = (\s_1\b_1/(\s_1^2 + \lambda^2),
\dots,\s_r\b_r/(\s_r^2 + \lambda^2), 0, \dots)^T\in
\R^q.\]
Thus, as $\lambda\to 0$,
\begin{align*}
\abs{x_{\textup{LS}} - x_\lambda}
& =\abs{\xi_{\textup{LS}} - \xi_\lambda } =
\Bigl|\Big(\frac{\lambda^2\b_1}{\s_1(\s_1^2+\lambda^2)}, \dots,
\frac{\lambda^2\b_r}{\s_r(\s_r^2+\lambda^2)}\Big)\Bigr|\longrightarrow 0.
\end{align*}
This concludes the  proof of \eqref{elleesse}.

\footnotesize

 \phantomsection
\addcontentsline{toc}{section}{References}  
  
\def\cprime{$'$}
\providecommand{\bysame}{\leavevmode\hbox to3em{\hrulefill}\thinspace}
\providecommand{\MR}{\relax\ifhmode\unskip\space\fi MR }
\providecommand{\MRhref}[2]{%
  \href{http://www.ams.org/mathscinet-getitem?mr=#1}{#2}
}
\providecommand{\href}[2]{#2}


\normalsize
\bigskip \noindent\sc \small  Annamaria Montanari, Daniele Morbidelli
\\ Dipartimento di Matematica,
Universit\`{a} di Bologna  (Italy)
\\Email: \tt   annamaria.montanari@unibo.it,
daniele.morbidelli@unibo.it

\end{document}